\documentclass[12pt]{amsart}
\usepackage{amsmath,latexsym,amsfonts,amssymb,amsthm}
\usepackage{geometry}
\usepackage{hyperref}
\usepackage{tabularx}
\usepackage{mathrsfs}
\usepackage{graphicx,color}
\usepackage[alphabetic]{amsrefs}
\usepackage[all,cmtip]{xy}
\usepackage{bbm}

\newtheorem{prop}{Proposition}[section]
\newtheorem{thm}[prop]{Theorem}

\newtheorem{conj}[prop]{Conjecture}
\newtheorem{lem}[prop]{Lemma}

\theoremstyle{definition}
\newtheorem{que}[prop]{Question}
\newtheorem{defn}[prop]{Definition}
\newtheorem{expl}[prop]{Example}
\newtheorem{rem}[prop]{\it Remark}

\newtheorem*{claim*}{Claim}

\newcommand{\bP}{\mathbb{P}}
\newcommand{\bC}{\mathbb{C}}
\newcommand{\bR}{\mathbb{R}}
\newcommand{\bA}{\mathbb{A}}
\newcommand{\bQ}{\mathbb{Q}}
\newcommand{\bZ}{\mathbb{Z}}
\newcommand{\bN}{\mathbb{N}}

\newcommand{\bG}{\mathbb{G}}
\newcommand{\bT}{\mathbb{T}}
\newcommand{\bk}{\mathbbm{k}}

\newcommand{\oX}{\overline{X}}

\newcommand{\cX}{\mathcal{X}}

\newcommand{\cO}{\mathcal{O}}

\newcommand{\cJ}{\mathcal{J}}

\newcommand{\cR}{\mathcal{R}}
\newcommand{\cS}{\mathcal{S}}

\newcommand{\fa}{\mathfrak{a}}

\newcommand{\fm}{\mathfrak{m}}
\newcommand{\ft}{\mathfrak{t}}

\newcommand{\fab}{\fa_{\bullet}}

\newcommand{\Spec}{\mathbf{Spec}}
\newcommand{\Proj}{\mathbf{Proj}}

\newcommand{\Hom}{\mathrm{Hom}}
\newcommand{\mult}{\mathrm{mult}}
\newcommand{\lct}{\mathrm{lct}}
\newcommand{\Ex}{\mathrm{Ex}}

\newcommand{\Aut}{\mathrm{Aut}}

\newcommand{\vol}{\mathrm{vol}}
\newcommand{\Vol}{\mathrm{Vol}}
\newcommand{\ord}{\mathrm{ord}}
\newcommand{\Val}{\mathrm{Val}}

\newcommand{\QM}{\mathrm{QM}}
\newcommand{\hvol}{\widehat{\rm vol}}

\newcommand{\wt}{\mathrm{wt}}

\newcommand{\mld}{\mathrm{mld}}
\newcommand{\mldk}{\mathrm{mld}^{\mathrm{K}}}

\newcommand{\an}{\mathrm{an}}
\newcommand{\Bl}{\mathrm{Bl}}
\newcommand{\LC}{\mathrm{LC}}

\newcommand{\gr}{\mathrm{gr}}
\newcommand{\Ric}{\mathrm{Ric}}

\numberwithin{equation}{section}

\title{Stability of klt singularities}

\author{Ziquan Zhuang}

\address{Department of Mathematics, Johns Hopkins University, Baltimore, MD 21218, USA}
\email{zzhuang@jhu.edu}

\date{}

\begin{document}

\maketitle

\begin{abstract}
    We survey some recent development in the stability theory of klt singularities. The main focus is on the solution of the stable degeneration conjecture.
\end{abstract}

\section{Introduction}

As Goresky and MacPherson put it in their famous monograph \cite{GM-Morse}*{p.26}, ``Philosophically, any statement about the projective variety or its embedding really comes from a statement about the singularity at the point of the cone. Theorems about projective varieties should be consequences of more general theorems about singularities which are no longer required to be conical''. In this expository article, we discuss this local and global correspondence in the context of K-stability, and survey some recent development in the local aspect of the stability theory.

\subsection{Motivation}

The local stability theory originates from questions in complex geometry. Recall that a K\"ahler-Einstein metric on a complex manifold is a K\"ahler metric $\omega$ with constant Ricci curvature. After appropriate rescaling, this means $\Ric(\omega)=\lambda\omega$ where $\lambda\in \{0,-1,1\}$. On a Fano manifold, we have $\lambda=1$. Consider a sequence of K\"ahler-Einstein Fano manifolds $(X_k,\omega_k)$ ($k=1,2,\dots$). By the convergence theory of Riemannian manifolds, specifically Gromov's compactness theorem, one can pass to a subsequence and extract a Gromov-Hausdorff limit $X_\infty$. In this context, Donaldson and Sun \cites{DS-GHlimit-1,DS-GHlimit-2} prove that the limit space $X_\infty$ is also a K\"ahler-Einstein Fano variety. In particular, it is algebraic (but may be singular). To analyze the singularities of $X_\infty$, they inspect the metric tangent cones of $X_\infty$, which are pointed Gromov-Hausdorff limits of $x\in (X_\infty,r_k \omega_\infty)$ for some fixed $x\in X_\infty$ and some increasing sequence of scaling factors $r_k\to\infty$. They find that the metric tangent cone again inherits some algebraic structure: it is a normal affine algebraic variety endowed with an effective torus action and a singular Ricci-flat K\"ahler cone metric. They also give a two-step degeneration description of the metric tangent cone, where the intermediate step (the K-semistable degeneration) is algebraic as well.

There should be an algebro-geometric explanation for the ubiquity of algebraic structures in these constructions, and this is achieved by the algebraic K-stability theory. The recent development in the (global) K-stability theory of Fano varieties culminates in the K-moduli theorem, which (among other things) provides an algebro-geometric construction of the Gromov-Hausdorff limit $X_\infty$. For those interested in this part of the theory, we recommend the survey \cite{Xu-K-stability-survey} and the upcoming book \cite{Xu-K-book} for a comprehensive and up-to-date account\footnote{The missing parts in \cite{Xu-K-stability-survey} are the properness of the K-moduli space and the surrounding higher rank finite generation theory. These topics are covered in \cite{Xu-K-book}.}. The local K-stability theory, which is the main topic of this survey article, will address Donaldson and Sun's conjecture that the two-step degeneration of $x\in X_\infty$ to its metric tangent cone should only depend on the algebraic structure of the singularity (rather than the metric). More generally, as we will explain in subsequent sections, every Kawamata log terminal (klt) singularity has a two-step degeneration to a uniquely determined K-polystable Fano cone singularity, and it seems likely that there is a K-moduli of klt singularities.


\subsection{History}

Apart from Donaldson-Sun's conjecture mentioned above, another source of inspiration for the development of the local stability theory is the question on the existence of Sasaki-Einstein metrics. In \cites{MSY-SE-toric,MSY-SE-general}, Martelli, Sparks and Yau set up a variational problem on Sasaki-Einstein manifolds whose critical point determines the Reeb vector field of the Sasaki-Einstein metric. The volume functional they considered and the minimization phenomenon they discovered may be seen as the first prototype of the local stability theory. Later, Collins and Sz\'ekelyhidi \cites{CS-Kss-Sasaki,CS-Sasaki-Einstein} proved a Yau-Tian-Donaldson type criterion for the existence of a Ricci flat K\"ahler cone metric on an isolated cone singularity, or equivalently, a Sasaki-Einstein metric on the link of the singularity. In particular, they defined K-semi/polystability for Fano cones (by mimicking the definitions in the global case \cites{Tian-K-stability-defn,Don-K-stability-defn}), and related the existence of a Ricci flat K\"ahler cone metric to the algebro-geometric condition that the singularity is a K-polystable Fano cone.

The algebraic theory of local K-stability starts with Chi Li's introduction of the normalized volumes of valuations \cite{Li-normalized-volume}. Li's insight (partly inspired by Martelli-Sparks-Yau's work) is that valuations on the singularity represent algebraic ``rescalings'' of the singularity, and that the valuation with the smallest normalized volume represents an ``optimal rescaling'' that should be closely related to the metric tangent cone degeneration. Based on this philosophy, he proposed to attack Donaldson-Sun's conjecture by solving a series of conjectures regarding the minimizer of the normalized volume function.

The theory is further investigated in \cites{LX-stability-kc,LX-stability-higher-rank}. In particular, Li and Xu \cite{LX-stability-higher-rank} show that the K-semistable degeneration step in Donaldson-Sun's construction only depends on the algebraic structure of singularity, and is indeed induced by a minimizer of the normalized volume function. Later \cite{LWX-metric-tangent-cone} completes the proof of Donaldson-Sun's conjecture by proving the algebraicity of the other step (i.e. the K-polystable degeneration) of the metric tangent cone construction.

The proof \cite{LX-stability-higher-rank} of the algebraicity of the K-semistable degenerations assumes the existence of such degenerations, which in turn relies on deep analytic results \cites{CCT-sing-GH-limit,Tia-partial-C^0,DS-GHlimit-1,DS-GHlimit-2} and as such is restricted to singularities on Gromov-Hausdorff limits of K\"ahler-Einstein Fano manifolds. To give a purely algebraic construction of the two-step degeneration, and to extend the theory to arbitrary klt singularities, \cite{LX-stability-higher-rank} refines Li's original proposal, and put forth what is now called the \emph{Stable Degeneration Conjecture} (see Section \ref{ss:SDC statement}). It highlights a number of conjectural properties of the normalized volume minimizer, which, when put together, ensure that every klt singularity has a canonical stable degeneration induced by the said minimizer.

The Stable Degeneration Conjecture is subsequently proved by a series of works: the existence of the normalized volume minimizer is prove by Blum \cite{Blu-minimizer-exist}, the uniqueness is established in \cite{XZ-minimizer-unique} (later \cite{BLQ-convexity} gives another proof), Xu \cite{Xu-quasi-monomial} proves that the minimizer is quasi-monomial (in \emph{loc. cit.} he also gives another proof that the minimizer exists), while the finite generation property (known by itself as the local higher rank finite generation conjecture) is confirmed in \cite{XZ-SDC}. It is proved in \cite{LX-stability-higher-rank} that the induced degeneration is a K-semistable degeneration, and \cite{LWX-metric-tangent-cone} further gives a recipe for constructing the K-polystable degeneration. These complete the algebro-geometric construction of the two-step degeneration.

The development of the local stability theory intertwines with the study on the K-stability of Fano varieties. The local and the global theory often draw inspirations from each other. The uniqueness of the normalized volume minimizer implies (through the cone construction) that equivariant K-semistability of a Fano variety is equivalent to K-semistability, and the proof of the uniqueness \cite{XZ-minimizer-unique} is in turn inspired by the earlier work on equivariant K-stability of Fano varieties \cite{Z-equivariant}. The idea behind Xu's proof \cite{Xu-quasi-monomial} of the quasi-monomial property of the minimizer led to the proof of the openness of K-semistability in families of Fano varieties \cites{Xu-quasi-monomial,BLX-openness}. The finite generation part of the Stable Degeneration Conjecture is a local analog of the higher rank finite generation conjecture for Fano varieties, proved in \cite{LXZ-HRFG}. In the setting of Fano varieties, there is also an algebro-geometric construction of canonical two-step degenerations to Fano varieties with K\"ahler-Ricci solitons \cite{BLXZ-soliton}.

Inspired by the K-moduli theory of Fano varieties, the focus of the local stability theory recently shifts towards the boundedness of singularities, an important missing ingredient for the local K-moduli theory. This topic has been intensively studied in \cites{HLQ-vol-ACC,MS-bdd-toric,LMS-bdd-dim-3,Z-mld^K-1,Z-mld^K-2}, yet the general case remains wide open.

\subsection{Outline}

Here is a roadmap for this survey. In section \ref{s:SDC}, we define some basic objects in the local stability theory and state the Stable Degeneration Conjecture. In Section \ref{s:kc}, we introduce the notion of Koll\'ar components, which plays an important role in the study of klt singularities. The entire Section \ref{s:geometry of minimizer} is devoted to explaining some key ingredients in the proof of the Stable Degeneration Conjecture. Section \ref{s:bdd} surveys our current understanding on the boundedness of klt singularities. Finally we discuss some conjectures and open questions in Section \ref{s:question}.

Since our primary focus is on the stable degeneration and the boundedness of klt singularities, we have to leave out several other interesting topics such as the analytic aspect of the theory and further applications of the normalized volume. Some of these topics have been covered by the survey \cite{LLX-nv-survey}, which we recommend to the interested readers.

\subsection{Notation and conventions}

We always work over an algebraically closed field $\bk$ of characteristic $0$. A singularity $x\in X$ consists of a normal variety $X$ and a closed point $x\in X$. We will often assume that $X$ is affine and will freely shrink $X$ around $x$ as needed.

\subsection*{Acknowledgement}

The author is partially supported by the NSF Grants DMS-2240926, DMS-2234736 and a Clay research fellowship. He would like to thank Harold Blum, Chi Li, Yuchen Liu, Xiaowei Wang and Chenyang Xu for many helpful comments and conversations.

\section{Stable Degeneration} \label{s:SDC}

The main result of the local stability theory is that every klt singularity has a canonical two-step stable degeneration induced by the valuation that minimizes the normalized volume. In this section, we elaborate the content of this statement. 

\subsection{Valuation} \label{ss:valuation}

We start with the notion of valuations. 

\begin{defn}
A (real) valuation on a variety $X$ is a map $v\colon \bk(X)^*\to \bR$ (where $\bk(X)$ denotes the function field of $X$), satisfying:
\begin{itemize}
 \item $v(fg)=v(f)+v(g)$;
 \item $v(f+g)\geq \min\{v(f),v(g)\}$;
 \item $v(\bk^*)=0$.
\end{itemize}
By convention, we set $v(0)=+\infty$.
\end{defn} 

Let us explain how valuations naturally arise in our context, at least in hindsight. For singularities appearing on Gromov-Hausdorff limits of K\"ahler-Einstein Fano manifolds, the stable degenerations are supposed to algebro-geometrically recover Donaldson-Sun's two-step degeneration description of the metric tangent cones. Now there does exist a tangent cone construction in algebraic geometry: for any singularity $x\in X=\Spec(R)$, its tangent cone is defined as $X_0=\Spec(R_0)$ where
\begin{equation} \label{e:naive tangent cone}
    R_0:=\bigoplus_{k\in \bN} \fm_x^k / \fm_x^{k+1}.
\end{equation}
As a typical example, if $X=(f=0)\subseteq \bA^{n+1}$ is a hypersurface singularity of multiplicity $k$ at the origin, and we write
\[
f=f_k+(\mbox{terms of multiplicity}\ge k+1)
\]
where $f_k$ is homogeneous of degree $k$, then the tangent cone to $0\in X$ is the hypersurface singularity $(f_k=0)\subseteq \bA^{n+1}$. On the other hand, it is not hard to see that this is not the desired metric tangent cone in general. One reason is that the tangent cone can be reducible, while the metric tangent cone is always irreducible (it is an affine \emph{variety}). In fact, the tangent cone of a (klt) hypersurface singularity $(f=0)\subseteq \bA^{n+1}$ coincides with its metric tangent cone if and only if $(f_k=0)\subseteq \bP^n$ is a K-polystable Fano variety (see Proposition \ref{prop:hypsurf sing minimizer}). By the Yau-Tian-Donaldson correspondence, the latter condition is equivalent to the existence of a K\"ahler-Einstein metric on the Fano variety.

What are some variations of the ``na\"ive'' tangent cone construction? The first observation is that the same construction can be applied to any decreasing graded sequence of $\fm_x$-primary ideals. 

\begin{defn}[\cite{JM-val-ideal-seq}]
A graded sequence of ideals in a ring $R$ is a sequence of ideals $\fa_\bullet=(\fa_k)_{k\in \bN}$ such that $\fa_0=R$ and $\fa_m \fa_n\subseteq \fa_{m+n}$. We say it is decreasing if $\fa_{k+1}\subseteq \fa_k$ for all $k\in\bN$.
\end{defn}

Given a decreasing graded sequence $\fab$ of $\fm_x$-primary ideals on $X=\Spec(R)$, we can form the associated graded algebra
\[
\gr_{\fab} R:=\bigoplus_{k\in \bN} \fa_k / \fa_{k+1}.
\]
When $\fa_k=\fm_x^k$, this recovers the graded algebra \eqref{e:naive tangent cone} that defines the tangent cone. In general, if the algebra $\bigoplus_{k\in \bN} \fa_k$ is finitely generated, then so is $\gr_{\fab} R$, and we get an isotrivial degeneration of $X$ to $\Spec(\gr_{\fab} R)$ through the Rees construction. To see this, set $\fa_k=R$ for all $k<0$ and let
\[
\cR:=\bigoplus_{k\in \bZ} \fa_k t^{-k} \subseteq R[t].
\]
Then one can check that $\cX=\Spec(\cR)\to \bA^1_t$ is a flat family with general fiber $X$ and special fiber $\Spec(\gr_{\fab} R)$. The $\bZ$-grading of $\cR$ also induces a $\bG_m$-action on the total space $\cX$ that commutes with the usual $\bG_m$-action on $\bA^1$. Such a family is also called a \emph{test configuration} of the singularity $x\in X$. The $\fm_x$-primary condition further ensures that the closed point $x\in X$ on the general fiber specializes to a closed point in the central fiber (the fixed point of the $\bG_m$-action). For us, such isotrivial degenerations serve as the algebraic analog of Gromov-Hausdorff limits.

As a slight generalization, we also allow graded sequences of ideals that are indexed by $\bR_{\ge 0}$. These are called filtrations. A formal definition is as follows.

\begin{defn}
A filtration of a ring $R$ is a collection of ideals $\fab=(\fa_\lambda)_{\lambda\in \bR_{\ge 0}}$ that is
\begin{enumerate}
    \item decreasing: $\fa_\lambda\subseteq \fa_\mu$ when $\lambda\ge \mu$,
    \item multiplicative: $\fa_\lambda \fa_\mu \subseteq \fa_{\lambda+\mu}$,
    \item left-continuous: $\fa_{\lambda-\varepsilon}=\fa_\lambda$ when $\lambda>0$ is fixed and $0<\varepsilon\ll 1$, and
    \item exhaustive: $\fa_0 = R$,  $\bigcap_{\lambda\ge 0} \fa_\lambda = \{0\}$.
\end{enumerate}
\end{defn}

Denote by $\Val_X$ the set of valuations that has a center on $X$, and by $\Val_{X,x}$ the set of valuations centered at $x\in X=\Spec(R)$\footnote{We say that a valuation $v$ is centered at a scheme-theoretic point $\eta\in X$ if there is a local inclusion $\cO_{X,\eta}\hookrightarrow\cO_v$ into the valuation ring $\cO_v:=\{f\in \bk(X)\mid v(f)\geq 0\}$ of $v$.}. Every valuation $v\in \Val_{X,x}$ induces an $\fm_x$-primary filtration $\fab(v)$ by setting
\[
\fa_\lambda(v):=\{f\in R\mid v(f)\ge \lambda\}.
\]

Similar to the case of graded sequences of ideals, for any filtration defined above we can form the associated graded algebra $\gr_{\fab} R:=\bigoplus_{\lambda\in \bR_{\ge 0}} \fa_\lambda / \fa_{>\lambda}$, where $\fa_{>\lambda}=\bigcup_{\mu>\lambda} \fa_\mu$. If $\fab=\fab(v)$ for some valuation $v\in \Val_{X,x}$, we further denote $\gr_{\fab} R$ as $\gr_v R$. With this level of generality, Donaldson and Sun \cite{DS-GHlimit-2} show that the two-step degenerations to the metric tangent cones are both induced by some $\fm_x$-primary filtration. Usually the first step is called the K-semistable degeneration, while the second step is the K-polystable degeneration. Philosophically, the two steps can be seen as analogous to the Harder-Narasimhan and Jordan-H\"older filtrations of vector bundles, where the graded pieces are (slope) semistable and polystable bundles, respectively. For now, we focus on the K-semistable degeneration, deferring the discussion of the K-polystable degeneration to Section \ref{ss:SDC statement}.

Since the filtration in \cite{DS-GHlimit-2} that induces the K-semistable degeneration is constructed using the (singular) K\"ahler-Einstein metric on the Gromov-Hausdorff limit, an immediate question, especially if we want to generalize the construction to arbitrary klt singularities, is how to identify the filtration using algebraic geometry. A simple but crucial observation is that the filtration is necessarily induced by some valuation, as the central fiber of the K-semistable degeneration is irreducible.

\begin{lem} \label{lem:irred degen -> valuation}
Let $\fab$ be an $\fm_x$-primary filtration of $R$ such that $\gr_{\fab} R$ is an integral domain. Then $\fab=\fab(v)$ for some valuation $v\in \Val_{X,x}$.
\end{lem}

\begin{proof}
For any $0\neq f\in R$ we let $v(f):=\sup\{\lambda\ge 0\mid f\in \fa_\lambda\}$. Note that $v(f)<+\infty$ since $\fab$ is exhaustive. Left-continuity of the filtration implies that $f\in \fa_{v(f)}$, while by definition $f\not\in \fa_{>v(f)}$. The condition that $\gr_{\fab} R$ is an integral domain translates into the equality $v(fg)=v(f)+v(g)$. From here it is clear that $v$ defines a valuation and $\fab=\fab(v)$.
\end{proof}

Because of this fact, we may restrict our search to valuations. Before we discuss further constraints, there are two important classes of valuations we shall keep in mind.

\begin{expl}[divisorial valuations]
Consider a proper birational morphism (such as a log resolution) $\pi\colon Y\to X$ where $Y$ is normal and let $E\subseteq Y$ be a prime divisor. We call such a divisor $E$ a prime divisor over $X$. Then we get a valuation $\ord_E$ which assigns to each $f\in \bk(X)^*=\bk(Y)^*$ its order of zero (or pole) along $E$. Rescalings of such valuations (i.e. $\lambda\cdot \ord_E$ for some $\lambda>0$) are called divisorial valuations. Note that the center of $\ord_E$ is the generic point of $\pi(E)$; in particular, the valuation $\ord_E$ is centered at $x\in X$ if and only if $E\subseteq\pi^{-1}(x)$.
\end{expl}

\begin{expl}[quasi-monomial valuations]
A generalization of the above example is the class of quasi-monomial valuations. Consider a proper birational morphism $\pi\colon Y\to X$ as above and a reduced divisor $E$ with irreducible components $E_1,\dots,E_r$. Assume that $Y$ is smooth and $E$ is simple normal crossing (SNC) at a generic point $\eta$ of $\cap_{i=1}^r E_i$. Then we have local coordinates $y_1,\dots,y_r$ such that $E_i=(y_i=0)$ around $\eta\in Y$. Any $f\in \cO_{Y,\eta}$ has a Taylor expansion
\[
f=\sum c_{\beta}y^{\beta} \in \widehat{\mathcal{O}}_{Y,\eta}\cong \bk(\eta)[\![y_1,\dots,y_r ]\!].
\]
For any $\alpha=(\alpha_1,\dots,\alpha_r)\in \bR_{\ge 0}^r\setminus\{0\}$, we can thus define a valuation $v_{\eta,\alpha}$ (or simply denoted as $v_\alpha$) by setting
\[
v_{\eta,\alpha}(f)=\min \left\{\langle \alpha,\beta \rangle \,|\, c_{\beta}\neq 0\right\}.
\]
In other words, it calculates the $\alpha$-weighted multiplicity of $f$. Such valuations are called quasi-monomial valuations\footnote{The name stems from the fact that the valuation $v_\alpha$ is monomial with respect to the local coordinates $y_1,\dots,y_r$ on the birational model $Y$.}. The rational rank of a quasi-monomial valuation $v_\alpha$ is defined as the dimension of the $\bQ$-vector space 
\[\mathrm{span}_\bQ \{\alpha_1,\dots,\alpha_r\}\subseteq \bR.\]
Equivalently, it is the rank of the value group $\Gamma=v_\alpha(\bk(X)^*)\subseteq \bR$. Note that $v_\alpha$ is divisorial if and only if its rational rank is one.
\end{expl} 

For a fixed pair $(Y,E)$ and a generic point $\eta$ of some stratum, we denote the corresponding set of quasi-monomial valuations by $\QM_\eta(Y,E)$. We also set $\QM(Y,E)=\cup_\eta \QM_\eta (Y,E)$, where $\eta$ varies over the smooth points of $Y$ at which $E$ is SNC. A general valuation can be thought of as limits of quasi-monomial valuations, c.f. \cite{JM-val-ideal-seq}*{Section 4}. In fact, for any log resolution $\pi\colon Y\to X$ and any SNC divisor $E\subseteq Y$ (we will henceforth call such pair $(Y,E)$ a log smooth model of $X$), there is a natural retraction map $r_{Y,E}\colon \Val_X\to \QM(Y,E)$. As we vary the model $(Y,E)$, the images under the retraction maps give approximations of a given valuation.

\subsection{Fano cone singularities} \label{ss:fano cone}

The singularities that appear on Gromov-Hausdorff limits of K\"ahler-Einstein Fano manifolds are examples of Kawamata log terminal (klt) singularities. From an algebro-geometric perspective, it is more natural to set up a local stability theory for klt singularities. This class of singularities is particularly important in birational geometry, as they are also singularities of minimal models of algebraic varieties. The output of the stable degenerations  belong to a special class of klt singularities called Fano cone singularities. We next review some basics on this class of singularities. For the readers' convenience, we first recall some definitions for singularities of pairs. More details can be found in \cite{KM98}.

\begin{defn}
A pair $(X,D)$ consisting of a normal variety $X$ and an effective $\bQ$-divisor $D$ is said to be klt (resp. log canonical, or lc for short) if $K_X+D$ is $\bQ$-Cartier and for any log resolution $\pi\colon Y\to X$ we have
\[
K_Y=\pi^*(K_X+D)+\sum a_i E_i
\]
where the $E_i$'s are the components of $\pi_*^{-1}D+\Ex(\pi)$ and $a_i>-1$ (resp. $a_i\ge -1$). A singularity $x\in X$ is klt (resp. lc) if $X$ is klt (resp. lc) around $x$.
\end{defn}

In a very rough sense, the klt (resp. lc) condition says that the singularities of holomorphic $n$-forms ($n=\dim X$) on $X$ are better (resp. not worse) than poles.

It will be convenient to reformulate the above definition using log discrepancies. For any pair $(X,D)$ and any prime divisor $E$ on some log resolution $\pi\colon Y\to X$, the log discrepancy $A_{X,D}(E)$ is defined to be
\[
A_{X,D}(E):=1+\ord_E(K_Y-\pi^*(K_X+D)).
\]
Then the pair $(X,D)$ is klt (resp. lc) if and only if $A_{X,D}(E)>0$ (resp. $\ge 0$) for all prime divisor $E$ over $X$.

\begin{rem}
For simplicity, we only state results in the context of klt singularities in this survey, but it is worth pointing out that the entire local stability theory also works for klt pairs $x\in (X,D)$.
\end{rem}

Heuristically, klt singularities are the local analog of Fano varieties:

\begin{expl}[orbifold cones, \cite{Kol13}*{Section 3.1}] \label{exp:cone-defn}
Cones over Fano manifolds are typical examples of klt singularities. More generally, for any projective variety $V$ and any ample $\bQ$-Cartier Weil divisor $L$ such that $L\sim_\bQ -rK_V$ for some $r>0$, the orbifold cone singularity
\[
o\in C_a(V,L):=\Spec \bigoplus_{m\in\bN} H^0(V,mL)
\]
is klt if and only if $V$ is a Fano variety with only klt singularities. 
\end{expl}

A more general construction that will play a key role in the local stability theory is given by Fano cone singularities. By definition, these are klt singularities with a nontrivial good torus action, together with the choice of a Reeb vector (also called a polarization). Let us explain the terminology. 

We say a torus $\bT=\bG_m^r$-action on a singularity $x\in X=\Spec(R)$ is {\it good} if it is effective and $x$ is in the orbit closure of any $\bT$-orbit. Let $N:=N(\bT)=\Hom(\bG_m, \bT)$ be the co-weight lattice and $M=N^*$ the weight lattice. We have a weight decomposition 
\[
R=\oplus_{\alpha\in M} R_\alpha,
\]
and the action being good implies that $R_0=\bk$ and every $R_\alpha$ is finite dimensional. For $f\in R$, we denote by $f_\alpha$ the corresponding component in the above weight decomposition.

\begin{defn}
A \emph{Reeb vector} on $X$ is a vector $\xi\in N_\bR$ such that $\langle \xi, \alpha \rangle>0$ for all $0\neq \alpha\in M$ with $R_{\alpha}\neq 0$. The set $\ft^+_{\bR}$ of Reeb vectors is called the Reeb cone\footnote{The terminologies are borrowed from contact geometry: suppose that $x\in X\subseteq \bC^N$ is an isolated singularity, then the link $L(x,X)=X\cap \{|z|=\varepsilon\}\subseteq \bC^N$ ($0<\varepsilon\ll 1$) is a contact manifold, and Reeb vectors on $X$ (in our definition) induce Reeb vector fields on the link (in the sense of contact geometry).}.
\end{defn}

For later use, we also define the notion of toric valuations. For any singularity $x\in X=\Spec(R)$ with a good torus action as above and any $\xi\in \ft^+_{\bR}$, we define a valuation $\wt_\xi$ (called a toric valuation) by setting
\[
\wt_\xi (f):=\min\{\langle \xi, \alpha \rangle\mid \alpha\in M, f_\alpha\neq 0\}
\]
where $f\in R$. It is not hard to verify that $v:=\wt_\xi\in \Val_{X,x}$. We also see that $\gr_v R\cong R$, as both sides have the same weight decomposition. In other words, the toric valuation $v$ induces a degeneration of the singularity to itself.

A Fano cone singularity will be denoted as $x\in (X;\xi)$ where $\xi$ is the Reeb vector field. Through the inclusion $\bT\subseteq \Aut(x,X)$, we often view the Reeb vector $\xi$ as an element of the Lie algebra of $\Aut(x,X)$. The subtorus in $\bT$ generated by $\xi$ is independent of $\bT$, and can be characterized as the smallest torus in $\Aut(x,X)$ whose Lie algebra contains $\xi$. If we assume that the torus $\bT$ is generated by the Reeb vector $\xi$ (and we will often do), then we may recover $\bT$ from the data $x\in (X;\xi)$. This justifies the absence of $\bT$ in the notation. We will denote the torus generated by $\xi$ as $\langle \xi \rangle$.

Let us describe two extreme cases of Fano cone singularities in more details.

\begin{expl}[toric singularities]
Every toric singularity is given by a strongly convex rational polyhedral cone $\sigma\subseteq N_\bR$ (see e.g. \cite{Fulton-toric}), and the Reeb cone $\ft^+_{\bR}$ is the interior of $\sigma$. The singularity is klt if and only if it is $\bQ$-Gorenstein. If this is the case, we get a Fano cone singularity after fixing a Reeb vetor $\xi\in \mathrm{Int}(\sigma)$. 
\end{expl}

\begin{expl}[quasi-regular Fano cones] \label{exp:quasi-reg cone defn}
A Fano cone singularity $x\in (X;\xi)$ is quasi-regular if $\langle \xi \rangle\cong \bG_m$, i.e. $\xi$ generates a one parameter subgroup. In this case the weight decomposition becomes $R=\oplus_{m\in\bN} R_m$. We may form the Proj and get $V:=\Proj(R)$. The natural projection $X\setminus \{x\}\to V$ is a Seifert $\bG_m$-bundle in the sense of \cite{Kol-Seifert-bundle}; in particular, for every closed point of $V$, the $\bG_m$-action on the corresponding reduced fiber is isomorphic to the left $\bG_m$-action on $\bG_m/\mu_r$ for some positive integer $r$. This gives rise to an orbifold boundary $\Delta_V=\sum_r (1-\frac{1}{r})\Delta_r$ where $\Delta_r\subseteq V$ is the divisorial part of the locus where the $\bG_m$-action on the reduced fiber has stabilizer $\mu_r$. By the local calculation in \cite{Kol-Seifert-bundle}*{Section 4} (which generalizes \cite{Kol13}*{Section 3.1}), we know that the pair $(V,\Delta_V)$ is klt and log Fano (i.e., $-(K_V+\Delta_V)$ is ample).
\end{expl}

As we will see in Proposition \ref{prop:tc by kc}, every klt singularity has a degeneration by test configuration to some Fano cone singularity (the proof relies on the notion of Koll\'ar components). The local stability theory will allow us to find the ``optimal'' degeneration.



\subsection{Normalized volume}

Chi Li observes that the K-semistable degeneration from Donaldson-Sun's construction is induced by a valuation that minimizes what he calls the \emph{normalized volume}. The definition involves two more classical invariants of valuations: the log discrepancy and the volume.

\begin{defn}[log discrepancy]
For any klt singularity $x\in X$, the \emph{log discrepancy} function
\[
A_X\colon \Val_X\to (0,+\infty],
\]
is defined as follows (c.f. \cite{JM-val-ideal-seq} and \cite{BdFFU-log-discrepancy}*{Theorem 3.1}). 
\begin{enumerate}
    \item For divisorial valuations $\lambda\cdot \ord_E$ where $E$ is a divisor over $X$, we set 
    \[
    A_X(\lambda\cdot \ord_E):=\lambda \cdot A_X(E).
    \]
    \item For quasi-monomial valuations $v_\alpha\in \QM(Y,E)$ where $(Y,E=\sum_{i=1}^r E_i)$ is a log smooth model and $\alpha\in \bR_{\ge 0}^r\setminus \{0\}$, we set
    \[
    A_X(v_\alpha):=\sum_{i=1}^r \alpha_i A_X(E_i).
    \]
    When $v_\alpha$ is divisorial, this recovers the previous definition.
    \item For general valuations $v\in \Val_{X,x}$, we set
    \[
    A_X(v):=\sup_{(Y,E)} A_X\left(r_{Y,E}(v)\right)
    \]
    where the supremum runs over all log smooth models of $X$, and $r_{Y,E}\colon \Val_X\to \QM(Y,E)$ is the retraction map discussed at the end of Section \ref{ss:valuation}.
\end{enumerate}
\end{defn}

It can happen that $A_X(v)=+\infty$ for some valuation $v$. We denote by $\Val^*_{X,x}$ the set of valuations $v\in\Val_{X,x}$ with  $A_X(v)<+\infty$. 

\begin{defn}[volume]
For any graded sequence $\fab$ of $\fm_x$-primary ideals, the volume of $\fab$ is defined as
\[
\vol(\fab) :=  \limsup_{m \to \infty} \frac{  \mathrm{length} ( \cO_{X,x}/\fa_m )}{m^n /n!}
\]
where $n=\dim X$. A similar invariant is the multiplicity of $\fab$, which is defined as 
\[
\mult(\fab) = \lim_{m\to \infty} \frac{ \mult(\fa_m)}{m^n}.
\]
In the geometric setting we consider, we have
\[
\vol(\fab) = \mult(\fab)
\]
by \cites{ELS03, Mus-mult-ideal-seq, LM-Okounkov-body, Cut13}. The \emph{volume} of a valuation $v\in\Val_{X,x}$ is defined as
\[
\vol(v)=\vol_{X,x}(v):=\vol(\fab(v))=\mult(\fab(v)).
\] 
\end{defn}

A basic observation is that both log discrepancy and volume are homogeneous in the variable: if we rescale the valuation $v$ to $\lambda v$, we find
\[
\vol(\lambda v)=\lambda^{-n}\vol(v),\,\,\mathrm{and}\,\,A_X(\lambda v)=\lambda\cdot A_X(v).
\]
It follows that $A_X(v)^n\cdot \vol(v)$ is invariant under rescaling. 

\begin{defn}[normalized volume \cite{Li-normalized-volume}]
Let $x\in X$ be an $n$-dimensional klt singularity. For any $v\in \Val^*_{X,x}$, we define the \emph{normalized volume} of $v$ as 
\[
\hvol(v)=\hvol_X(v):=A_X(v)^n\cdot\vol(v).
\]
By convention, we also set $\hvol(v)=+\infty$ when $A_X(v)=+\infty$. The \emph{local volume} of the singularity $x\in X$ is defined as
\[
  \hvol(x,X):=\inf_{v\in\Val^*_{X,x}} \hvol_X(v).
\]
\end{defn}

\begin{rem}
In some literature, the valuation space $\Val_{X,x}$ is called the non-archimedean link of $x\in X$, since it can be thought of as a punctured neighbourhood of $x^{\an}$ in the Berkovich analytification $X^{\an}$ of $X$. We can form the normalized non-archimedean link $NL(x,X)$ as the quotient $\Val_{X,x}/\bR_+$ where $\bR_+$ acts by rescaling. Since normalized volume function is rescaling invariant, it descends to a function on the normalized non-archimedean link.
\end{rem}

The local volume of a klt singularity can also be computed using normalized multiplicities of ideals, as observed in \cite{Liu-vol-sing-KE}. This alternative approach offers a great deal of flexibility in the study of this invariant.

\begin{thm}[\cite{Liu-vol-sing-KE}*{Theorem 27}] 
For any klt singularity $x\in X$ of dimension $n$, we have
\begin{equation} \label{e:Liu's inequality}
    \hvol(x,X) = \inf_{\fa} \lct(\fa)^n \cdot \mult(\fa) = \inf_{\fab} \lct(\fab)^n \cdot \mult(\fab)
\end{equation}
where the first $($resp. second$)$ infimum runs over all $\fm_x$-primary ideals $($resp. graded sequences of ideals$)$.    
\end{thm}

Here $\lct(\fa)$ is the log canonical threshold (lct) of the ideal $\fa$, defined as
\[
\lct(\fa):=\inf_{v\in \Val_{X,x}^*} \frac{A_X(v)}{v(\fa)}.
\]
It is also the largest number $\lambda>0$ such that $(X,\fa^\lambda)$ is log canonical. The log canonical threshold of a graded sequence $\fab$ of ideals is defined in a similar manner, replacing $\fa$ by $\fab$ in the above formula. By \cite{JM-val-ideal-seq}, the infimum is in fact a minimum.

The proof of the formula \eqref{e:Liu's inequality} is quite straightforward. On one hand, we have $A_X(v)\ge \lct(\fab(v))$ and $\vol(v)=\mult(\fab(v))$ for any valuation $v\in \Val_{X,x}^*$, hence
\[
\hvol(v)\ge \lct(\fab(v))^n \cdot \mult(\fab(v)).
\]
On the other hand, for any valuation $v\in \Val_{X,x}^*$ that computes $\lct(\fab)$ we may rescale it so that $v(\fab)=1$. By definition this gives $\lct(\fab)=A_X(v)$ and $\fab\subseteq \fab(v)$, hence $\mult(\fab)\ge \vol(v)$ and
\[
\lct(\fab)^n \cdot \mult(\fab)\ge \hvol(v).
\]

It turns out that the local volume of a klt singularity is always positive \cite{Li-normalized-volume} (we will sketch a proof in Section \ref{s:kc}) and thus becomes an interesting invariant of the singularity. If $x\in X$ lives on some Gromov-Hausdorff limit of K\"ahler-Einstein Fano manifolds, then the local volume $\hvol(x,X)$ has the following differential geometric interpretation (see \cite{LLX-nv-survey}*{Theorem 5.6}):
\[
\frac{\hvol(x,X)}{\hvol(0,\bA^n)} = \lim_{r\to 0} \frac{\Vol(B_r(x,X))}{r^{2n}\Vol(B_1(0,\bA^n))},
\]
where the right hand side is the volume density (in the sense of geometric measure theory) of the K\"ahler-Einstein limit metric. An interesting question is the distribution of the possible values of local volumes, see Conjecture \ref{conj:ACC}. 

A guiding principle of the local stability theory, put forward by Li \cite{Li-normalized-volume}, is that the K-semistable degeneration of a klt singularity is induced by the valuation with the smallest normalized volume. For singularities on Gromov-Hausdorff limits of K\"ahler-Einstein Fano manifolds, this is confirmed in \cite{LX-stability-higher-rank}*{Section 3.1}. Here we illustrate this connection through a few examples.

\begin{expl}[smooth point] \label{exp:nv sm pt}
Consider $0\in X=\bA^n$ and let $v_\alpha$ ($\alpha\in \bR^n_+$) be a monomial valuation with respect to the coordinates $x_1,\dots,x_n$. We have $A_X(v_\alpha)=\alpha_1+\dots+\alpha_n$ and $\vol(v_\alpha)=(\alpha_1\dots\alpha_n)^{-1}$ by direct calculations, thus
\begin{equation*}
    \hvol(v_\alpha) = \frac{(\alpha_1+\dots+\alpha_n)^n}{\alpha_1\dots\alpha_n}.
\end{equation*}
In particular, we see that $\hvol(v_\alpha)\ge n^n$, with equality if and only if all the weights $\alpha_i$ are equal, i.e. $v_\alpha = c\cdot \mult_0$ for some $c>0$. It is slightly harder to compute the local volume of a smooth point using the valuative definition. Instead we resort to normalized multiplicities \eqref{e:Liu's inequality}. Using toric degeneration, it is shown in \cite{dFEM-mult-and-lct} that 
\[
\lct(\fa)^n \cdot \mult(\fa)\ge n^n
\]
for any $\fm_x$-primary ideal $\fa$ when $x\in X$ is smooth. This implies that $\hvol(0,\bA^n)=n^n$ and that $\mult_0$ is a minimizer of the normalized volume function.
\end{expl}

\begin{expl}[toric singularities]
The argument in the above example can be generalized to show that on any klt toric singularity $x\in X$ the normalized volume function is minimized by some toric valuation $v_\xi$, where $\xi\in N_\bR$; we leave the details to the reader. From the discussions in Section \ref{ss:fano cone}, we know that the toric minimizer induces a degeneration of the toric singularity to itself. This is compatible with the differential geometric picture: the toric singularity admits a Ricci-flat K\"ahler cone metric, and the metric tangent cone is the toric singularity itself. Moreover, the vector field on $X$ that gives the homothetic scaling along the rays of the K\"ahler cone is naturally identified with $\xi\in N_{\bR}$, see \cites{MSY-SE-toric,FOW-toric-Sasaki}.
\end{expl}

\begin{expl}[cone singularities] \label{exp:cone minimizer}
Consider an orbifold cone singularity $o\in X:=C_a(V,L)$ as in Example \ref{exp:cone-defn}. The exceptional divisor of the orbifold blowup at $o$ gives a divisorial valuation $v=\ord_o$ on $X$. It is also characterized by the condition that it is invariant under the natural $\bG_m$-action and that $v(s)=m$ for all $s\in H^0(V,mL)$. If $L$ is Cartier and sufficiently ample, then we also have $\fa_k(v)=\fm_o^k$ and hence $v$ induces the degeneration to the tangent cone (which in this case is isomorphic to $o\in X$ itself). However, it is not always the case that $v$ minimizes the normalized volume function: it is proved in \cites{Li-equivariant,LL-vol-minimizer-KE,LX-stability-kc} that this happens if and only if the Fano variety $V$ is K-semistable. The latter is a necessary condition for $X$ to admit a Ricci-flat K\"ahler cone metric \cite{CS-Kss-Sasaki}. This gives another strong evidence that the minimizing valuations of the normalized volume function contains rich information about the local stability of the singularities.
\end{expl}

The local volumes of klt singularities also enjoy some nice properties. We only list some of them here, referring to \cite{LLX-nv-survey} for a more thorough discussion.

\begin{thm}[lower semi-continuity, \cite{BL-vol-lsc}] \label{thm:vol lsc}
For any $\bQ$-Gorenstein family $B\subseteq \cX\to B$ of klt singularities, the function
\[
b\in B\mapsto \hvol(b,\cX_b)
\]
on $B$ is lower semi-continuous with respect to the Zariski topology.
\end{thm}

Here we call $B\subseteq \cX\to B$ a $\bQ$-Gorenstein family of klt singularities if $\cX$ is flat over $B$, $B\subseteq \cX$ is a section of the projection, $K_{\cX/B}$ is $\bQ$-Cartier and $b\in \cX_b$ is klt for any $b\in B$.

\begin{thm}[largest volume, \cite{LX-cubic-3fold}*{Appendix}] \label{thm:largest vol}
For any klt singularity $x\in X$ of dimension $n$, we have $\hvol(x,X)\le n^n$, with equality if and only if $x\in X$ is smooth.
\end{thm}

Note that the inequality part is also a consequence of the lower semi-continuity of local volumes, but the equality case requires more work.

\begin{prop}[behavior under birational morphism, \cite{LX-cubic-3fold}*{Lemma 2.9}] \label{prop:vol under birational map}
Let $\pi\colon Y\to X$ be a proper birational morphism between klt varieties. Assume that $K_Y\le \pi^* K_X$. Then $\hvol(y,Y)\ge \hvol(x,X)$ for any $x\in X$ and any $y\in \pi^{-1}(x)$.
\end{prop}

In particular, local volumes are non-increasing under small birational morphisms. On the other hand, it is less clear how they behave under flips.

\subsection{Stable Degeneration Conjecture} \label{ss:SDC statement}

We now introduce the Stable Degeneration Conjecture, which gives a recipe for constructing the K-semistable degenerations of klt singularities using the minimizers of the normalized volume function.

\begin{conj}[\cites{Li-normalized-volume,LX-stability-higher-rank}] \label{conj:SDC}
Let $x\in X=\Spec(R)$ be a klt singularity. Then:
\begin{enumerate}
    \item \emph{(Existence of minimizer).} There exists a valuation $v_0\in \Val_{X,x}$ such that \[\hvol(v_0)=\hvol(x,X).\]
    \item \emph{(Uniqueness).} The normalized volume minimizer $v_0$ is unique up to rescaling.
    \item \emph{(Quasi-monomial).} The minimizer $v_0$ is a quasi-monomial valuation.
    \item \emph{(Finite generation).} The associated graded algebra $\gr_{v_0} R$ is finitely generated.
    \item \emph{(Stability).} The quasi-monomial minimizer $v_0$ induces a natural Reeb vector $\xi_0$ on $X_0:=\Spec(\gr_{v_0}R)$, and $x_0\in (X_0;\xi_0)$ is a K-semistable Fano cone singularity.
\end{enumerate}
\end{conj}

Let us elaborate the various parts of the above conjecture. Since the K-semistable degeneration of a klt singularity eventually comes from the minimizing valuation of the normalized volume function, the existence of the minimizer is a necessary condition to begin with. The uniqueness part can be reformulated as saying that the normalized volume function has a unique minimizer on the normalized non-archimedean link. It implies the uniqueness of the K-semistable degeneration, since rescaling the valuation does not change the isomorphism class of the associated graded algebra.

Assuming that there exists a unique minimizer $v_0$, the natural candidate of the K-semistable degeneration (as we have discussed in Section \ref{ss:valuation}) is $\Spec(\gr_{v_0} R)$. But there is a serious issue here, since a priori the algebra $\gr_{v_0} R$ need not be finitely generated. An obvious necessary condition is that the value semigroup $v_0(R\setminus\{0\})$ is finitely generated. With a bit more work, one can show that it is also necessary that the minimizer $v_0$ is a quasi-monomial valuation. This justifies the third item of the conjecture.

Unfortunately, there are still many quasi-monomial valuations whose associated graded algebras are not finitely generated, see Example \ref{exp:non f.g.}. The finite generation part (also called the local higher rank finite generation conjecture) of the Stable Degeneration Conjecture turns out to be quite subtle.

Taking (1)--(4) for granted, let us elaborate the precise content of item (5). First we need to explain where the Reeb vector comes from. Denote by $r$ the rational rank of the quasi-monomial minimizer $v_0$. By choosing a (non-canonical) isomorphism $\Gamma\cong \bZ^r$, we may replace the $\Gamma$-grading on $\gr_{v_0} R$ by a $\bZ^r$-grading. In particular, we get a $\bT=\bG^r_m$-action on $\gr_{v_0} R$. Since $v_0$ takes the same positive value on each $\gr^\lambda_{v_0} R$, it induces a toric valuation on $\gr_{v_0} R$ and hence a Reeb vector $\xi_0$ on $X_0=\Spec(\gr_{v_0}R)$. The grading also determines a closed point $x_0$ that is the unique closed orbit of the torus action, thus we get a Fano cone singularity $x_0\in (X_0;\xi_0)$.

Next we shall define K-semistability for Fano cone singularities. The original definition from \cites{CS-Kss-Sasaki,CS-Sasaki-Einstein} is via the non-negativity of generalized Futaki invariants. We choose the following definition, which is more convenient for our purpose.

\begin{defn}
We say a Fano cone singularity $x\in (X;\xi)$ is \emph{K-semistable} if 
\[
\hvol(x,X)=\hvol_X (\wt_\xi),
\]
i.e. the toric valuation $\wt_\xi$ minimizes the normalized volume.
\end{defn}

Its equivalence with the original definition is shown in \cite{LX-stability-higher-rank}*{Theorem 2.34}. Intuitively, the generalized Futaki invariants of the Fano cone singularity are ``directional derivatives'' of the normalized volume function at $\wt_\xi$, hence they are non-negative if $\wt_\xi$ is a minimizer. There is also a local-to-global correspondence: by \cites{Li-equivariant,LL-vol-minimizer-KE,LX-stability-kc} (see the discussions in Example \ref{exp:cone minimizer}), a cone singularity $o\in C_a(V,L)$ is K-semistable if and only if the Fano base $V$ is K-semistable. 

The stable degeneration of a klt singularity is a two-step process. Conjecture \ref{conj:SDC} takes care of the first step, the K-semistable degeneration. The other step, the K-polystable degeneration, can be done using the following theorem.

\begin{thm}[\cite{LWX-metric-tangent-cone}*{Theorem 1.2}] \label{thm:K-ps degeneration}
Given a K-semistable Fano cone singularity $x_0\in (X_0;\xi_0)$, there always exists a special test configuration that degenerates $x_0\in (X_0;\xi_0)$ to a K-polystable Fano cone singularity $y\in (Y;\xi_Y)$. Moreover, such a K-polystable degeneration $y\in (Y;\xi_Y)$ is uniquely determined by $x_0\in (X_0;\xi_0)$ up to isomorphism.
\end{thm}

Let us clarify some of the terminologies in the above statement. 

\begin{defn}
A special test configuration of a klt singularity $x\in X$ is a test configuration with klt central fiber. 

A special test configuration of a Fano cone singularity $x\in (X;\xi)$ is a $\bT=\langle\xi\rangle$-equivariant special test configuration of the klt singularity $x\in X$. The central fiber is also a Fano cone singularity $y\in (Y;\xi_Y)$\footnote{In fact using the fiberwise $\bT$-action we can identify $\xi_Y$ with $\xi$ in $N(\bT)_\bR$.}. If it is K-semistable, we call it a K-semistable degeneration of $x\in (X;\xi)$\footnote{This should not be confused with the K-semistable degeneration of the klt singularity $x\in X$ in the Stable Degeneration Conjecture.}.
\end{defn}

Next we define K-polystability. Again, the original definition involves generalized Futaki invariants, but we choose the following more convenient definition. They are equivalent by \cite{LWX-metric-tangent-cone}. By \cites{CS-Kss-Sasaki,CS-Sasaki-Einstein,Hua-thesis}, we also know that a Fano cone singularity is K-polystable if and only if it admits a Ricci flat K\"ahler cone metric.

\begin{defn}
We say a Fano cone singularity $x\in (X;\xi)$ is \emph{K-polystable} if it is K-semistable, and any K-semistable degeneration is isomorphic to $x\in (X;\xi)$.
\end{defn}

The intuition behind this definition is a notion of $S$-equivalence: two semistable objects are considered $S$-equivalent if one of them isotrivially degenerates to the other, and polystable objects are the ones without any further $S$-equivalent degenerations\footnote{The definition of polystable vector bundle and GIT-polystable point can both be formulated this way: two semistable vector bundles are $S$-equivalent if they have the same Jordan-H\"older factors, and a vector bundle is polystable if it is a direct sum of its Jordan-H\"older factors; similarly in GIT (geometric invariant theory), two GIT-semistable points are $S$-equivalent if their orbit closure intersect, and the GIT-polystable point represents the unique closed orbit in this $S$-equivalence class.}.

The proofs of Conjecture \ref{conj:SDC} and Theorem \ref{thm:K-ps degeneration} will be sketched in Section \ref{s:geometry of minimizer}. 

\section{Koll\'ar components} \label{s:kc}

In this section, we highlight an important tool in the study of klt singularities: Koll\'ar components. This notion was originally introduced in \cite{Xu-pi_1-finite} to study the local fundamental groups of klt singularities (see also \cites{Pro-plt-blowup,Kud-plt-blowup} for some precursors), and has since found many other applications. While the cone construction provides one direction of the local-to-global correspondence, Koll\'ar components work in the opposite direction: it often helps to reduce questions about klt singularities to questions about Fano varieties. In the K-stability context, Koll\'ar components also serve as the local analog of special test configurations \cite{LX-stc}, which play a key role in the K-stability theory of Fano varieties.

\begin{defn}[Koll\'ar component] \label{def:kc}
Let $x\in X$ be a klt singularity and let $E$ be a prime divisor over $X$. If there exists a proper birational morphism $\pi\colon Y\to X$ such that $\pi$ is an isomorphism away from $x$, $E=\pi^{-1}(x)$, $(Y,E)$ is plt and  $-(K_Y+E)$ is $\pi$-ample, we call $E$ a \emph{Koll\'ar component} over $x\in X$ and call $\pi\colon Y\to X$ the plt blowup of $E$.
\end{defn}

Intuitively, a Koll\'ar component is the exceptional divisor of a partial resolution that is also a Fano variety. In fact, by adjunction (c.f. \cite{Kol13}*{Section 4.1}), we may write 
\begin{equation} \label{e:Delta_E}
    (K_Y+E)|_E = K_E+\Delta_E
\end{equation}
for some effective divisor $\Delta_E$ (called the different) on $E$, and the condition that $E$ is a Koll\'ar component implies that $(E,\Delta_E)$ is a klt log Fano pair. Since $K_Y+E=\pi^*K_X+A_X(E)\cdot E$ and $A_X(E)>0$, we also see that $-E$ is $\pi$-ample and this implies that the plt blowup is uniquely determined by the Koll\'ar component $E$.

\begin{expl}
If $x\in X$ is the orbifold cone over a klt Fano variety as in Example \ref{exp:cone-defn}, then the exceptional divisor of the orbifold blowup at the vertex $x$ is a Koll\'ar component. More generally, for any quasi-regular Fano cone singularity $x\in (X;\xi)$, the zero section of the corresponding Seifert $\bG_m$-bundle $X\setminus\{x\}\to V$ (see Example \ref{exp:quasi-reg cone defn}) is a Koll\'ar component over $x\in X$.
\end{expl}

By \cite{Xu-pi_1-finite}, every klt singularity has at least one Koll\'ar component. In fact, the proof in \emph{loc. cit.} shows that every log canonical threshold is computed by some Koll\'ar component. 

\begin{thm}[\cite{Xu-pi_1-finite}*{Lemma 1}] \label{thm:kc compute lct}
Let $x\in X$ be a klt singularity and let $\fa$ be an $\fm_x$-primary ideal. Then there exists a Koll\'ar component $E$ over $x\in X$ such that 
\[
\lct(\fa)=\frac{A_X(E)}{\ord_E(\fa)}.
\]
\end{thm}

The existence of Koll\'ar component already has the following consequence.

\begin{prop} \label{prop:tc by kc}
Every klt singularity has a degeneration by test configuration to some Fano cone singularity.
\end{prop}

\begin{proof}
Take a Koll\'ar component $E$ over $x\in X=\Spec(R)$ and consider the induced degeneration to $X_0:=\Spec(\gr_E R)$\footnote{This is just a shorthand notation for $\gr_{\ord_E}R$.}. It suffices to show that $X_0$ is a Fano cone singularity. For simplicity, assume that $E$ is Cartier on the associated plt blowup $\pi\colon Y\to X$. Then $\Delta_E=0$ in the adjunction formula \eqref{e:Delta_E} and
\[
\gr^m_E R\cong \pi_*\cO_Y(-mE)/\pi_*\cO_Y(-(m+1)E)
\]
can be identified with $H^0(E,-mE|_E)$, as the next term in the long exact sequence is 
\[
R^1\pi_*\cO_Y(-(m+1)E)=0
\]
by Kawamata-Viehweg vanishing. Hence $X_0$ is a cone over the klt Fano variety $E$. In the general case, the central fiber $X_0$ is only an orbifold cone over $(E,\Delta_E)$ polarized by the $\bQ$-line bundle $-E|_E$, but the basic idea is the same, see e.g. \cite{LZ-Tian-sharp}*{Proposition 2.10}.
\end{proof}

We can also use Koll\'ar component to show that the local volume of a klt singularity is always positive. This is originally proved in \cite{Li-normalized-volume}. The argument we present here is slightly different, but the main ideas are the same. The key ingredient is an Izumi type inequality: 

\begin{lem} \label{lem:Izumi}
Let $x\in X$ be a klt singularity. Then there exists some constant $C>0$ such that
\[
v(\fm_x)\ord_x\le v \le C \cdot A_X(v) \ord_x
\]
for any valuation $v\in \Val_{X,x}^*$.
\end{lem}

Here $\ord_x(f):=\max\{k\in \bN\,|\,f\in \fm_x^k\}$ for $f\in \cO_X$. To see why this lemma implies the positivity of the local volume, note that 
\[
\hvol(x,X)=\inf_{v:\,A_X(v)=1} \vol(v)
\]
by the rescaling invariance of $\hvol$. For such valuations $v$, the Izumi inequality above implies $v\le C\cdot \ord_x$ and hence $\vol(v)\ge C^{-n}\cdot \mult_x X$. This gives $\hvol(x,X)\ge C^{-n}\cdot \mult_x X>0$.

\begin{proof}[Proof of Lemma \ref{lem:Izumi}]
The first inequality is definitional. For the second inequality, it is enough to prove
\[
v\le C\cdot A_X(v)\ord_E
\]
for some Koll\'ar component $E$ over $x\in X$. We can reformulate this statement as $A_X(v)\ge v(D)$ for all $\bQ$-Cartier divisor $D$ on $X$ with $\ord_E(D)\le C^{-1}$. Thus the question is equivalent to finding some constant $\varepsilon>0$ such that $(X,D)$ is lc
whenever $\ord_E(D)\le \varepsilon$. On the plt blowup $\pi\colon Y\to X$ of $E$ we have
\[
\pi^*(K_X+D) \le  K_Y + \pi^{-1}_* D + E
\]
as long as $\varepsilon\le A_X(E)$. By inversion of adjunction (see e.g. \cite{Kol13}*{Theorem 4.9}), the pair $(X,D)$ is lc if and only if $(E,\Delta_E+\pi^{-1}_*D|_E)$ is lc. Since $(E,\Delta_E)$ is klt, we essentially reduce to a similar question in lower dimension. By induction on the dimension, we may assume there exists some $0<\varepsilon\ll 1$\footnote{In fact, we can choose $\varepsilon$ to be the alpha invariant \cites{Tia-alpha,CS-alpha-Fano3} of the log Fano pair $(E,\Delta_E)$. The proof we sketch here is essentially the proof that alpha invariants are positive.} such that $(E,\Delta_E+\Gamma)$ is lc for all effective $\bQ$-divisor $\Gamma\sim_\bQ -\varepsilon E|_E$. This gives the desired constant as $\pi^{-1}_*D\sim_\bQ -\varepsilon E$.
\end{proof}

\subsection{Divisorial minimizer} \label{ss:div minimizer}

Using Koll\'ar component, we now discuss a special case of the Stable Degeneration Conjecture, namely when the minimizer is a divisorial valuation. It is worth noting that, in general, the minimizer can be a valuation of higher rank; one such example is the cone over $\Bl_p \bP^2$, see \cite{Blu-minimizer-exist}*{Section 8.3}. Nevertheless, the divisorial case will provide some intuition for our understanding of the higher rank case.

The key observation is that divisorial minimizers are necessarily Koll\'ar components.

\begin{thm} \label{thm:div minimizer is kc}
Any divisorial minimizer of the normalized volume is of the form $v=\lambda\cdot\ord_E$ for some Koll\'ar component $E$.
\end{thm}

This is proved in \cite{Blu-minimizer-exist}*{Proposition 4.9} and \cite{LX-stability-kc}*{Theorem C} using somewhat different arguments. The first step is to show that the divisorial minimizer $v=\ord_E$ satisfies the finite generation property (Conjecture \ref{conj:SDC}(4)), which follows from two basic observations:
\begin{enumerate}
    \item If $v$ is a minimizer of $\hvol$, then $v$ is the unique valuation that computes $\lct(\fab(v))$. This can be derived from the equality conditions in the proof of \eqref{e:Liu's inequality}, see \cite{Blu-minimizer-exist}*{Lemma 4.7} for more details.
    \item If $v$ is a divisorial valuation that computes the log canonical threshold $\lct(\fab)$ of some graded sequence of ideals, then it satisfies the finite generation property. This is essentially a consequence of \cite{BCHM}*{Corollary 1.4.3}. See the proof of \cite{LX-stability-kc}*{Lemma 3.11}.
\end{enumerate}
The finite generation property ensures that there exists some sufficiently divisible integer $m$ such that $\fa_{mr}(v)=\fa_m(v)^r$ for all $r\in \bN$. From this we deduce that the divisor $E$ also computes $\lct(\fa_m)$ and it is the unique such divisor (by the item (1)  mentioned above). Since every log canonical threshold is computed by some Koll\'ar component (Theorem \ref{thm:kc compute lct}), we see that $E$ is a Koll\'ar component.

Once we know that the divisorial minimizer comes from a Koll\'ar component $E$, we can study the minimizer in terms of the geometry of the associated log Fano pair $(E,\Delta_E)$. The results can be summarized as follows:

\begin{thm}[\cite{LX-stability-kc}*{Theorem 1.2}] \label{thm:kc minimizer K-ss}
A Koll\'ar component $E$ over a klt singularity $x\in X$ minimizes the normalized volume if and only if the log Fano pair $(E,\Delta_E)$ is K-semistable. Moreover, such a K-semistable Koll\'ar component, if it exists, is unique.
\end{thm}

Using this theorem, we can now verify one of the facts mentioned in Section \ref{ss:valuation}.

\begin{prop} \label{prop:hypsurf sing minimizer}
Let $0\in (f=0)\subseteq \bA^{n+1}$ be a klt hypersurface singularity with tangent cone $(f_k=0)\subseteq \bA^{n+1}$. Then $\mult_0$ is a valuation that minimizes the normalized volume if and only if $(f_k=0)\subseteq \bP^n$ is a K-semistable Fano variety.
\end{prop}

\begin{proof}
Note that $\mult_0$ is a valuation if and only if $f_k$ is irreducible. In this case, the ordinary blowup at the origin gives an exceptional divisor $E\cong (f_k=0)\subseteq \bP^n$ with $\Delta_E=0$. The result then follows from the previous theorem.
\end{proof}

\section{Geometry of minimizers} \label{s:geometry of minimizer}

The Stable Degeneration Conjecture has been proved by the works \cites{Blu-minimizer-exist,LX-stability-higher-rank,Xu-quasi-monomial,XZ-minimizer-unique,XZ-SDC} (see also \cites{LX-stability-kc,LWX-metric-tangent-cone,BLQ-convexity}). In this section, we explain some main ideas of its proof; we will also sketch the proof of Theorem \ref{thm:K-ps degeneration} (the existence of K-polystable degeneration). Throughout, we fix a klt singularity $x\in X$.

\subsection{Existence} \label{ss:existence}

We first explain why volume minimizers exist, following \cite{Blu-minimizer-exist}. 

\begin{thm}[\cite{Blu-minimizer-exist}] \label{thm:minimizer exist}
For any klt singularity $x\in X$, there exists a valuation $v_0\in \Val^*_{X,x}$ that minimizes the normalized volume function.
\end{thm}

Take a sequence of valuations $v_k\in \Val_{X,x}^*$ ($k=1,2,\dots$) such that 
\[
\lim_{k\to \infty}\hvol(v_k)\to \hvol(x,X).
\]
We may also rescale the valuations so that $A_X(v_k)=1$ (this is necessary to force the valuations $v_k$ to lie in a compact subset of $\Val_{X,x}$). Ideally, we want to construct a minimizer $v_0$ as a limit of the sequence $v_1,v_2,\dots$. For such an argument to work, one would need to know that the normalized volume function is lower semi-continuous on the valuation space. Unfortunately, it is still an open question whether this is the case or not. 

Instead, we consider the graded sequences of valuation ideals $\fab(v_k)$. We already see in \eqref{e:Liu's inequality} that the normalized volume can also be computed using normalized multiplicities of graded ideal sequences. Moreover, by the proof of \eqref{e:Liu's inequality}, if we can find a graded sequence of ideals $\fab$ such that 
\begin{equation} \label{e:nv computed by ideal}
    \hvol(x,X) = \lct(\fab)^n\cdot \mult(\fab),
\end{equation}
then any valuation that computes $\lct(\fab)$ would be a minimizer of $\hvol$. 

Such a graded sequence $\fab$ is constructed in \cite{Blu-minimizer-exist} as a ``generic limit'' of the sequences $\fab(v_k)$ ($k=1,2,\dots$). The idea is to consider, for each $m\in \bN$, the locus $H_m$ in the Hilbert scheme that contains a Zariski dense subset which parametrizes the ``truncated'' graded sequences of ideals 
\begin{equation*} 
    \fa_m(v_k)\subseteq\dots\subseteq \fa_1(v_k).
\end{equation*}
There are natural truncation maps 
\[
\pi_{m+1}\colon H_{m+1}\to H_m.
\]
One can show (see \cite{Blu-minimizer-exist}*{Section 5}) that there exists a compatible sequence of closed points $x_m\in H_m$, where each $x_m$ is a very general point of $H_m$, such that $\pi_{m+1}(x_{m+1})=x_m$. They parametrize a graded sequence $\fab$ of ideals\footnote{We view it as a ``generic limit'' of the sequences $\fab(v_k)$ ($k=1,2,\dots$), since the limit point is obtained as a very general point of their Zariski closure in the Hilbert scheme.}, and the goal is to verify the identity \eqref{e:nv computed by ideal} for this $\fab$.

From the generic limit construction, we have
\[
\lct(\fa_m)^n\cdot \mult(\fa_m) = \limsup_{k\to\infty} \big( \lct(\fa_m(v_k))^n\cdot \mult(\fa_m(v_k))\big),
\]
since both functions $x_m \mapsto \lct(\fa_m)$ and $x_m\mapsto \mult(\fa_m)$ are constructible on the Hilbert scheme (in particular on $H_m$). By our choice of $v_k$, we also have
\[
\limsup_{k\to\infty} \big(\lct(\fab(v_k))^n\cdot \mult(\fab(v_k))\big) = \hvol(x,X).
\]
A moment of thought reveals that the missing ingredient is the following uniform convergence statement. 

\begin{prop} 
For any $\varepsilon>0$, there exists some positive integer $M$ such that
\[
\lct(\fa_m(v_k))^n\cdot \mult(\fa_m(v_k)) \le (1+\varepsilon)\cdot \lct(\fab(v_k))^n\cdot \mult(\fab(v_k)) 
\]
for all $m\ge M$ and all $k=1,2,\dots$.
\end{prop}

\begin{proof}
We always have $m\cdot \lct(\fa_m)\le \lct(\fab)$, so the main question is to show
\begin{equation} \label{e:mult uniform converge}
    \frac{\mult(\fa_m(v_k))}{m^n} \le (1+\varepsilon)\cdot \mult(\fab(v_k))
\end{equation}
for large $m$. The proof of this uses asymptotic multiplier ideals. Recall that for any graded sequence of ideals $\fab$ on $X$ and any rational number $c>0$, the asymptotic multiplier ideal $\cJ(c\cdot \fab)$ (see \cite{Laz-positivity-2}*{Section 11.1} and \cite{BdFFU-log-discrepancy}*{Theorem 1.2}) is the ideal on $X$ consisting of local sections $f\in \cO_X$ such that
\[
v(f) > c\cdot v(\fab) - A_X(v)
\]
for all valuations $v\in \Val^*_{X,x}$. To illustrate the ideas, let us first assume that $x\in X$ is smooth for simplicity. For any valuation $v\in \Val^*_{X,x}$ and any $m\in\bN$, the asymptotic multiplier ideals $\cJ(m\cdot \fab(v))$ of the corresponding sequence of valuation ideals satisfy
\[
\fa_m(v)\subseteq \cJ(m\cdot \fab(v)) \subseteq \fa_{m-A_X(v)}(v),
\]
where both inclusions follow from the definition of multiplier ideals. When $x\in X$ is smooth, the asymptotic multiplier ideals also satisfy subadditivity \cite{Laz-positivity-2}*{Theorem 11.2.3}, in particular,
\[
\cJ(m\ell\cdot \fab)\subseteq \cJ(m\cdot \fab)^\ell
\]
for any $m,\ell\in\bN$. A formal consequence of these two properties, when applied to the valuations $v_k$ (rescaled so that $A_X(v_k)=1$), is that\footnote{This is also the argument behind \cite{ELS03}*{Theorem A}.}
\[
\fa_{m\ell}(v_k)\subseteq \cJ(m\ell\cdot\fab(v_k))\subseteq \cJ(m\cdot \fab(v_k))^\ell \subseteq \fa_{m-1}(v_k)^\ell.
\]
From here it is not hard to deduce \eqref{e:mult uniform converge}.

When $x\in X$ is singular, we only have a weaker subadditivity result (see \cite{Takagi-multiplier-ideal}*{Theorem 0.1} or \cite{Eis-subadd-char0-pf}*{Theorem 7.3.4}):
\[
\mathrm{Jac}_X^\ell\cdot \cJ(m\ell\cdot \fab)\subseteq \cJ(m\cdot \fab)^\ell,
\]
where $\mathrm{Jac}_X$ is the Jacobian ideal of $X$. As before this gives
\[
\mathrm{Jac}_X^\ell\cdot \fa_{m\ell}(v_k) \subseteq \fa_{m-1}(v_k)^\ell.
\]
What is important to us is that the ``correction term'' $\mathrm{Jac}_X^\ell$ is independent of the valuation $v_k$, and its effect on the multiplicity is negligible when $m\to \infty$ (the precise proof uses Teissier’s Minkowski Inequality and Li's properness estimate). We may then conclude as in the smooth case. See \cite{Blu-minimizer-exist}*{Proposition 3.7} for the technical details.
\end{proof}

\begin{rem}
The generic limit argument we sketch above requires the base field to be uncountable, since we need to choose very general points of the locus $H_m$. Using boundedness of complements, \cite{Xu-quasi-monomial} gives another proof for the existence of minimizer that works for general fields. Alternatively, once we know that the minimizer is unique up to rescaling (Theorem \ref{thm:uniqueness}) and in particular invariant under the Galois action, we can first base change to an uncountable field to find a minimizer and then Galois descend to the original base field.
\end{rem}

\begin{rem}
The proof of the lower semi-continuity of local volumes (\cite{BL-vol-lsc}, see Theorem \ref{thm:vol lsc}) follows a similar circle of ideas, but carried out in families. Roughly speaking, since the local volumes can be approximated by normalized multiplicities of the form $\lct(\fa)^n\cdot \mult(\fa)$ and the log canonical threshold function is lower semi-continuous in families (a consequence of the inversion of adjunction), the main obstruction comes from the multiplicity term, which is not lower semi-continuous in families. In fact, it is the opposite: multiplicities usually increases under specialization. Nonetheless, one can extend the argument proving Theorem \ref{thm:minimizer exist} to show that the local volume can be uniformly approximated by the normalized colengths
\[
\lct(\fa)^n\cdot \ell(\cO_X/\fa)
\]
of ideals that are bounded below by some fixed power of $\fm_x$. Here in order to ensure that the approximation is uniform in families, one needs to show that the constants in the Izumi type inequality (Lemma \ref{lem:Izumi}) is uniformly bounded in families. As these constants ultimately rely on the Koll\'ar components, this can be done by extracting a family version of Koll\'ar components. Since the colength function $\fa\mapsto \ell(\cO_X/\fa)$ is locally constant on the Hilbert scheme, and the lct part is lower semi-continuous, the lower semi-continuity of local volumes is then a direct consequence of this uniform approximation result. 
\end{rem}

\subsection{Uniqueness and K-semistability}

In \cite{LX-stability-higher-rank}, it is shown that the quasi-monomial minimizer of the normalized volume function is unique\footnote{Strictly speaking, we only have uniqueness up to rescaling, but we will not write it out every time.}, under the assumption that the minimizer has a finitely generated associated graded algebra. The first proof of the uniqueness that is independent of the other parts of the Stable Degeneration Conjecture appears in \cite{XZ-minimizer-unique}, and later \cite{BLQ-convexity} finds another argument. Both proofs rely on a notion of geodesics between valuations, and ultimately the uniqueness of the minimizer can be seen a consequence of the ``geodesic (strong) convexity'' of the volume function.

Ideally, convexity means
\[
\vol((1-t)v_0+tv_1)\le (1-t)\cdot \vol(v_0)+t\cdot \vol(v_1)
\]
for any valuations $v_0,v_1\in \Val_{X,x}^*$, except that it is not clear how to make sense of the ``valuation'' $(1-t)v_0+tv_1$. On the other hand, there is a natural way to interpret $(1-t)v_0+tv_1$ as a \emph{filtration}: for any $\lambda\in \bR_+$, we take $\fa_{\lambda,t}$ to be the $\fm_x$-primary ideal generated by those $f\in \cO_{X,x}$ such that 
\[
(1-t)\cdot v_0(f)+t\cdot v_1(f) \ge \lambda.
\]
The reader may easily verify that this defines a filtration $\fa_{\bullet,t}$ for each $t\in [0,1]$, and that $\fa_{\bullet,0}$ (resp. $\fa_{\bullet,1}$) is the filtration $\fab(v_0)$ (resp. $\fab(v_1)$) induced by $v_0$ (resp. $v_1$). We view the family $(\fa_{\bullet,t})_{t\in [0,1]}$ of filtrations as the \emph{geodesic} between $v_0$ and $v_1$. More generally, given two filtrations $\fa_{\bullet,0}$ and $\fa_{\bullet,1}$, we can define \cites{XZ-minimizer-unique,BLQ-convexity}\footnote{See also \cites{BLXZ-soliton,Reb-geodesic} for the global version of this construction.} the geodesic between them as the following family $(\fa_{\bullet,t})_{t\in [0,1]}$ of filtrations:
\[
\fa_{\lambda,t} = \sum_{(1-t)\lambda_0+t\lambda_1=\lambda} \fa_{\lambda_0,0}\cap \fa_{\lambda_1,1}.
\]
In some sense, the space of filtrations is the ``geodesic completion'' of the valuation space $\Val_{X,x}^*$. We already have an extension of the normalized volume function to the space of filtrations using normalized multiplicities \eqref{e:Liu's inequality}, and the more natural question is whether the individual terms in \eqref{e:Liu's inequality} are convex along geodesics. This is confirmed by the following statement.

\begin{thm} \label{thm:convexity}
For any $t\in [0,1]$, we have
\begin{enumerate}
    \item$($\cite{XZ-minimizer-unique}*{Theorem 3.11}$)$
    \[
    \lct(\fa_{\bullet,t})\le (1-t)\cdot \lct(\fa_{\bullet,0})+t\cdot \lct(\fa_{\bullet,1}).
    \]
    \item$($\cite{BLQ-convexity}*{Theorem 1.1}$)$
    \[
    \mult(\fa_{\bullet,t})^{-1/n}\ge (1-t)\cdot \mult(\fa_{\bullet,0})^{-1/n}+t\cdot \mult(\fa_{\bullet,1})^{-1/n}.
    \]
    Moreover, equality holds if and only if there exists $c>0$ such that $v(\fa_{\bullet,0})=c\cdot v(\fa_{\bullet,1})$ for all valuations $v\in \Val_{X,x}^*$. 
\end{enumerate}
\end{thm}

The statement (1) is deduced from a summation formula of asymptotic multiplier ideals, while statement (2) relies on the construction of a two-dimensional Duistermaat-Heckman measure using a compatible basis with respect to the two filtrations $\fa_{\bullet,0}$ and $\fa_{\bullet,1}$. We refer to the original articles for the relevant details, here we just explain why this theorem implies the uniqueness of $\hvol$-minimizer.

\begin{thm}[\cites{XZ-minimizer-unique,BLQ-convexity}] \label{thm:uniqueness}
Up to rescaling, there is a unique valuation $v_0$ that minimizes the normalized volume function $\hvol$.
\end{thm}

Before we discuss the proof, let us mention one interesting consequence of this theorem: the finite degree formula for local volumes.

\begin{thm}[\cite{XZ-minimizer-unique}*{Theorem 1.3}] \label{thm:finite deg formula}
Let $f\colon (y\in Y)\to (x\in X)$ be a finite quasi-\'etale morphism between klt singularities. Then 
\[
\hvol(y,Y) = \deg(f) \cdot \hvol(x,X). 
\]
\end{thm}

More generally, the finite degree formula holds for crepant Galois morphisms, i.e., Galois morphisms $f\colon \big(y\in (Y,\Delta_Y)\big)\to \big(x\in(X,\Delta)\big)$ such that $f^*(K_X+\Delta)=K_Y+\Delta_Y$. Roughly, the reason for the finite degree formula is that the unique minimizer is necessarily invariant under the Galois action, hence descends to the quotient. Moreover, the normalized volume gets divided by $\deg(f)$ as the valuation descends.

We now return to the proof of Theorem \ref{thm:uniqueness}.

\begin{proof}[Proof of Theorem \ref{thm:uniqueness}]
Suppose we have two minimizers $v_0,v_1$. Consider the filtrations $\fa_{\bullet,i}=\fab(v_i)$ ($i=0,1$) and the geodesic $\fa_{\bullet,t}$ connecting them. By \eqref{e:Liu's inequality}, we have
\[
\hvol(x,X)^{1/n} \le \frac{\lct(\fa_{\bullet,t})}{\mult(\fa_{\bullet,t})^{-1/n}},
\]
hence using Theorem \ref{thm:convexity} we obtain
\begin{align*}
    \hvol(x,X)^{1/n} & \le \frac{
    (1-t)\cdot \lct(\fa_{\bullet,0})+t\cdot \lct(\fa_{\bullet,1})
    }{
    (1-t)\cdot \mult(\fa_{\bullet,0})^{-1/n}+t\cdot \mult(\fa_{\bullet,1})^{-1/n}
    } \\
    & \le \max\left\{ \frac{\lct(\fa_{\bullet,0})}{\mult(\fa_{\bullet,0})^{-1/n}},
    \frac{\lct(\fa_{\bullet,1})}{\mult(\fa_{\bullet,1})^{-1/n}}\right\} \\
    & \le \max\{\hvol(v_0)^{1/n}, \hvol(v_1)^{1/n}\} = \hvol(x,X)^{1/n}.
\end{align*}
Thus equality holds everywhere. In particular, one can show that the equality condition in Theorem \ref{thm:convexity}(2) implies $v_1=cv_0$ for some $c>0$.
\end{proof}

We remark that our presentation so far draws heavily from \cite{BLQ-convexity}. The original proof of Theorem \ref{thm:uniqueness} in \cite{XZ-minimizer-unique} exploits the K-semistability of the minimizing valuation rather than the full convexity of the volume function. This approach has some other interesting consequences; most notably, it gives the following generalization of Theorem \ref{thm:kc minimizer K-ss}. 

\begin{thm}[\cite{XZ-minimizer-unique}*{Theorems 3.7 and 3.10}] \label{thm:minimizer = K-ss val}
A valuation $v_0\in \Val_{X,x}^*$ minimizes the normalized volume function if and only if it is K-semistable.
\end{thm}

Let us provide a brief definition of the K-semistability of valuations, which mimics the characterization of K-semistability of Fano varieties using basis type divisors. Recall from \cites{FO-delta,BJ-delta} that an ($m$-)basis type divisor on a Fano variety $V$ is a divisor of the form
\[
D = \frac{1}{mN_m} \sum_{i=1}^{N_m} \{s_i=0\}
\]
where $N_m = h^0(V,-mK_V)$ and $s_1,\dots,s_{N_m}$ is a basis of $H^0(V,-mK_V)$ (typically we choose $m\gg 0$), and that a Fano variety $V$ is K-semistable if and only if its basis type divisors are ``asymptotically log canonical'', i.e.,
\begin{equation} \label{e:A>=S}
    A_V(v)\ge S_V(v):=\lim_{m\to\infty} \sup\{v(D)\,|\,D\mbox{ is of } m\mbox{-basis type}\}
\end{equation}
for all valuations $v$ on $V$.

Suppose next that we have a valuation $v\in \Val_{X,x}^*$ over the klt singularity $x\in X$. To define its K-semistability, we rescale it so that $A_X(v)=1$ and consider $m$-basis type divisors (with respect to $v$) of the form
\[
D = \frac{1}{mN_m} \sum_{i=1}^{N_m} \{s_i=0\},
\]
where this time 
\[
N_m=\dim(\fa_m(v)/\fa_{m+1}(v))
\]
and $s_1,\dots,s_{N_m}\in \fa_m(v)$ restrict to a basis of $\fa_m(v)/\fa_{m+1}(v)$. We say the valuation $v$ is K-semistable if its basis type divisors are ``asymptotically log canonical'', i.e., for any $w\in \Val_{X,x}^*$ we have
\[
A_X(w)\ge S(v;w):=\lim_{m\to\infty} \sup\{w(D)\,|\,D\mbox{ is of } m\mbox{-basis type}\}
\footnote{A priori, this limit may not exist. It is also not clear how the numbers $N_m$ grows. To avoid these potential subtleties, the actual definition of $S(v;w)$ in \cite{XZ-minimizer-unique} uses a slight variant of basis type divisors. For the simplicity of our presentation, we will not make this distinction.}.
\]
Note that we have $A_X(v)=S(v;v)=1$ by definition, hence if the valuation $v$ is K-semistable, then it is automatically an lc place of its own basis type divisors (in the asymptotic sense). If $E$ is a Koll\'ar component over $x\in X$, then it follows from inversion of adjunction (see \cite{XZ-minimizer-unique}*{Theorem 3.6}) that the divisorial valuation $\ord_E$ is K-semistable if and only if the induced log Fano pair $(E,\Delta_E)$ is K-semistable. Thus Theorem \ref{thm:kc minimizer K-ss} can be viewed as a special case of Theorem \ref{thm:minimizer = K-ss val}.

The proof of Theorem \ref{thm:minimizer = K-ss val} naturally divides into two steps. First we need to show that the minimizers of the normalized volume function are K-semistable. This is done by analyzing the derivatives of the normalized volume function along the geodesic connecting the minimizer $v_0$ to an arbitrary valuation $w\in \Val_{X,x}^*$. The non-negativity of the derivative at the minimizer $v_0$ is (almost) exactly the condition $A_X(w)\ge S(v_0;w)$ that defines K-semistability. To show the other direction, i.e. K-semistable valuations are $\hvol$-minimizers, one interprets the normalized volume as a ``log canonical threshold'' via the identity
\[
\hvol_X(v)^{1/n}=\frac{A_X(v)}{\vol(v)^{-1/n}}.
\]
A key step is to realize the denominator $\vol(v)^{-1/n}$ as the asymptotic vanishing order along the valuation $v$ of certain basis type divisors. From this perspective, the $\hvol$-minimizers are just the valuations that asymptotically compute the log canonical thresholds of basis type divisors. Since K-semistable valuations are exactly of this kind, they minimize the normalized volume. For more details, see \cite{XZ-minimizer-unique}.

Suppose for the moment that the minimizing valuation $v_0$ is quasi-monomial and has a finitely generated associated graded algebra $\gr_{v_0} R$ (these will be verified in the next two subsections). Then as we see in Section \ref{ss:SDC statement}, we have a degeneration of $x\in X=\Spec(R)$ to $x_0\in X_0:=\Spec(\gr_{v_0}R)$, and there is an induced Reeb vector $\xi_0$ on $X_0$. Li and Xu \cite{LX-stability-higher-rank} show that what we get is a K-semistable Fano cone singularity.

\begin{thm}[\cite{LX-stability-higher-rank}] \label{thm:SDC K-ss part}
The Fano cone singularity $x_0\in (X_0;\xi_0)$ is K-semistable, and the degeneration is volume preserving, i.e. $\hvol(x,X)=\hvol(x_0,X_0)$.
\end{thm}

These can also be explained using the K-semistability of the minimizer $v_0$. All we need to show is that the toric valuation $\wt_{\xi_0}$ on $X_0$ minimizes the normalized volume. By Theorem \ref{thm:minimizer = K-ss val}, this is equivalent to showing that $\wt_{\xi_0}$ is a K-semistable valuation. We know that $v_0$ is K-semistable since it is the normalized volume minimizer on $X$, hence its basis type divisors are asymptotically log canonical, and always have the valuation $v_0$ as an lc place. In general, given a log canonical pair, the degenerations induced by its lc places have semi log canonical central fibers, since the latter are orbifold cones over pairs coming from adjunction along the lc places, c.f. \cite{BLX-openness}*{Appendix A.1}. This essentially implies that the degenerations of the basis type divisors to $X_0$ remain asymptotically log canonical. It appears that what we get from these degenerations are exactly the basis type divisors on $X_0$ (with respect to $\wt_{\xi_0}$). Therefore, the toric valuation $\wt_{\xi_0}$ is K-semistable by definition.

\subsection{Quasi-monomial property}

We have seen in Section \ref{ss:existence} that the normalized volume minimizer computes the log canonical threshold of some graded sequence of ideals. Regarding such valuations, Jonsson and Musta\c{t}\u{a} have made the following conjecture.

\begin{conj}[\cite{JM-val-ideal-seq}]
Let $X$ be klt and let $\fab$ be a graded sequence of ideals on $X$ such that $\lct(\fab)<\infty$.
\begin{enumerate}
    \item\emph{(Weak version)}. There exists a quasi-monomial valuation that computes $\lct(\fab)$.
    \item\emph{(Strong version)}. Every valuation that computes $\lct(\fab)$ is quasi-monomial.
\end{enumerate}
\end{conj}

The strong version of this conjecture is still open. An important breakthrough in the development of the K-stability theory is Xu's proof \cite{Xu-quasi-monomial} of the weak version of Jonsson-Musta\c{t}\u{a}'s Conjecture. An immediate corollary is the quasi-monomial property of the $\hvol$-minimizer. In this subsection, we sketch the main ideas of this proof. 

\begin{thm}[\cite{Xu-quasi-monomial}*{Theorem 1.1}]
Let $\fab$ be a graded sequence of ideals on a klt variety $X$ such that $\lct(\fab)<\infty$. Then there exists a quasi-monomial valuation $v$ that computes $\lct(\fab)$.
\end{thm}

One way to get a valuation that computes $\lct(\fab)$ is to take an $m\to\infty$ limit of valuations $v_m$ that compute $\lct(\fa_m)$. The latter are always quasi-monomial. This is of course not surprising, since every valuation is a limit of quasi-monomial valuations. Therefore, the main difficulty is to control the limit process.

A crucial ingredient is the theory of complement. Recall that a ($\bQ$-)complement of an lc pair $(X,\Delta)$ is an effective $\bQ$-divisor $D\sim_\bQ -(K_X+\Delta)$ such that $(X,\Delta+D)$ is lc. It is called an $N$-complement (for some positive integer $N$) if $N(K_X+\Delta+D)\sim 0$. A valuation $v$ on $X$ is called an lc place of the complement $D$ if $A_{X,\Delta+D}(v)=0$; such valuations are always quasi-monomial. We use $\LC(X,\Delta+D)\subseteq \Val_X$ to denote the corresponding set of lc places.

A difficult theorem of Birkar \cite{Birkar-bab-1}*{Theorem 1.7}, known as the boundedness of complements, states that if $X$ is of Fano type and $(X,\Delta)$ admits a complement, then it also has an $N$-complement for some integer $N$ that only depends on the dimension of $X$ and the coefficient of $\Delta$. This has the following consequence.

\begin{prop}
Let $x\in X$ be a klt singularity. Then there exists some positive integer $N$ depending only on the dimension of $X$ such that for any $\fm_x$-primary ideal $\fa$, any valuation computing $\lct(\fa)$ is an lc place of some $N$-complement.
\end{prop}

Roughly speaking, this is because any valuation computing $\lct(\fa)$ is automatically an lc place of some complement (an obvious choice is a general member of the $\bQ$-ideal $\fa^{\lct(\fa)}$), hence by Birkar's theorem we can upgrade the complement to a bounded complement.

The proposition in particular applies to the valuations $v_m$ that compute $\lct(\fa_m)$. Because the integer $N$ does not depend on the ideal $\fa_m$, modulo some sufficiently large power of the maximal ideal $\fm_x$ we can further arrange that the valuations $v_m$ are lc places of a \emph{bounded} family of $N$-complements. It follows that the limit $\lim_{m\to\infty} v_m$ is not arbitrary: it is a generic limit in a bounded family of simplices of quasi-monomial valuations. From here, we conclude that the limit valuation stays in the same family of simplices; in particular, it is quasi-monomial. In fact, the proof naturally implies a stronger statement:

\begin{thm}[\cite{Xu-quasi-monomial}] \label{thm:lc place N-comp}
Let $\fab$ be a graded sequence of ideals such that $\lct(\fab)<\infty$. Then  $\lct(\fab)$ is computed by some lc place of $N$-complement, where the integer $N$ only depends on the dimension. 

In particular, the minimizing valuation of the normalized volume is an lc place of $N$-complement.
\end{thm}

By applying a similar technique in families, one can also show that local volumes of klt singularities are constructible in families. This is indeed an important ingredient in the proof of the openness of K-semistability in families of Fano varieties.

\begin{thm}[\cite{Xu-quasi-monomial}*{Theorem 1.3}] \label{thm:constructible}
For any $\bQ$-Gorenstein family $B\subseteq \cX\to B$ of klt singularities, the function
\[
b\in B\mapsto \hvol(b,\cX_b)
\]
on $B$ is constructible with respect to the Zariski topology.
\end{thm}

Since every $\hvol$-minimizer is an lc place of $N$-complement, the key point is to analyze how the volume changes as the $N$-complement varies. The constructibility statement in Theorem \ref{thm:constructible} is ultimately a consequence of a local version of the deformation invariance of log plurigenera \cite{HMX-BirAut} (see also \cite{Siu-plurigenera}).

\subsection{Finite generation}

We now come to finite generation part of the Stable Degeneration Conjecture, which is the main result of \cite{XZ-SDC}.

\begin{thm}[\cite{XZ-SDC}*{Theorem 1.1}] \label{thm:HRFG}
Let $x\in X=\Spec(R)$ be a klt singularity and let $v_0$ be the minimizer of the normalized volume function $\hvol$ on $\Val_{X,x}^*$. Then the associated graded algebra $\gr_{v_0} R$ is finitely generated.
\end{thm}

Instead of proving finite generation for this particular valuation, we will describe a finite generation criterion for more general valuations. To motivate such a criterion, we first revisit the argument in the divisorial case.

We have seen in Section \ref{ss:div minimizer} that if the minimizer $v_0$ is divisorial, then the associated graded algebra $\gr_{v_0} R$ is finitely generated. This can also be deduced from Theorem \ref{thm:lc place N-comp}, as \emph{divisorial} lc places of complements satisfy the finite generation property by \cite{BCHM}\footnote{The reason is that the graded algebra associated to a divisorial valuation can be written as the quotient of a certain Cox ring. When the divisor is an lc place of complement, this Cox ring is finitely generated by \cite{BCHM}.}. Since the minimizer is still an lc place of complement in the higher (rational) rank situation (Theorem \ref{thm:lc place N-comp}), one may ask:

\begin{que}
Is it true that $\gr_v R$ is finitely generated for any valuation $v\in \Val_{X,x}^*$ that is an lc place of complement?
\end{que}

Unfortunately the answer is no. Indeed, the global version of this question already has a negative answer \cites{AZ-K-adjunction,LXZ-HRFG}.

\begin{expl} \label{exp:non f.g.}
Any valuation $v$ on a projective variety $V$ induces a filtration of the section ring of an ample line bundle $L$. It is proved in \cite{LXZ-HRFG}*{Theorem 4.5} that when $V$ is Fano, $L$ is proportional to $-K_V$ and $v$ is an lc place of complement, the induced filtration is finitely generated if and only if the $S$-invariant function defined in \eqref{e:A>=S} is locally linear on the rational envelope of $v$, i.e. a simplex $\QM(Y,E)$ of smallest dimension that contains $v$. The latter condition is automatic for any divisorial valuation (since the rational envelope is just a single point), but gets highly non-trivial for higher rank valuations. It already fails for some lc places of a nodal cubic curve $C\subseteq \bP^2$, see \cite{LXZ-HRFG}*{Section 6}. By the cone construction, this provides plenty of valuations over $0\in \bA^3$ that are lc places of complements but the associated graded algebras are not finitely generated. In fact, one can even write down a simplex $\QM(Y,E)$ of lc places of complements such that $v\in \QM(Y,E)$ satisfies the finite generation property if and only if $v$ is divisorial.
\end{expl}

Recall that any two valuations $v_0,v_1$ are connected by a geodesic $(\fa_{\bullet,t})_{0\le t\le 1}$ in the space of filtrations. If both $v_0$ and $v_1$ are divisorial lc places of the same complement, then using \cite{BCHM} it is not too hard to show that the filtrations $\fa_{\bullet,t}$ along the geodesic all have finitely generated associated algebras. On the other hand, we can draw lines in any simplex $\QM(Y,E)$. The failure of the finite generation property of higher rank valuations essentially comes from the fact that these lines are not necessarily geodesics in the valuation space.

We need to find additional properties of the $\hvol$-minimizer that turn lines in its rational envelope into geodesics. We do know more about the divisorial minimizers: they are also induced by Koll\'ar components (Theorem \ref{thm:div minimizer is kc}). The following higher rank analog turns out to be a key to the proof of Theorem \ref{thm:HRFG}.

\begin{defn}[\cite{XZ-SDC}*{Definition 3.7}]
Let $x\in X$ be a klt singularity. A \emph{Koll\'ar model} of $x\in X$ is a birational model $\pi\colon (Y,E)\to X$ such that $\pi$ is an isomorphism away from $\{x\}$, $E=\pi^{-1}(x)$, $(Y,E)$ is dlt\footnote{The readers may notice that the definition in \cite{XZ-SDC} uses the qdlt (shorthand for quotient of dlt) condition rather than the dlt one. The difference is not essential, except that the qdlt version would make some technical steps easier. Since we will only focus on the more conceptual part of the proof, we ignore the difference and work with the dlt version.} and $-(K_Y+E)$ is ample.
\end{defn}

The only difference with the definition of Koll\'ar components (Definition \ref{def:kc}) is that we allow the exceptional divisor $E$ to have more than one components, i.e., we drop the rank one condition.

\begin{defn}
Let $x\in X$ be a klt singularity. We say a quasi-monomial valuation $v\in \Val_{X,x}^*$ is a \emph{Koll\'ar valuation} if there exists a Koll\'ar model $\pi\colon (Y,E)\to X$ such that $v\in \QM(Y,E)$.
\end{defn}

Using Koll\'ar models and the notion of  Koll\'ar valuations, we can finally formulate the finite generation criterion.

\begin{thm}[\cite{XZ-SDC}*{Theorem 4.1}] \label{thm:fg criterion via km}
Let $x\in X=\Spec(R)$ be a klt singularity, and let $v\in \Val_{X,x}^*$ be a quasi-monomial valuation. Then the following are equivalent.
\begin{enumerate}
    \item The graded algebra $\gr_v R$ is finitely generated and $X_v:=\Spec(\gr_v R)$ is klt.
    \item The valuation $v$ is a Koll\'ar valuation.
\end{enumerate}
\end{thm}

We only sketch the proof of the implication $(2)\Rightarrow (1)$, since this is what we need for the finite generation part of the Stable Degeneration Conjecture (Theorem \ref{thm:HRFG}). Let $\pi\colon (Y,E)\to X$ be a Koll\'ar model such that $v\in \QM(Y,E)$. For simplicity, let us assume that $E$ only has two components $E_0$ and $E_1$. According to what we discuss before, the finite generation of $\gr_v R$ would follow if we can show that the geodesic $(\fa_{\bullet,t})_{t\in [0,1]}$ joining $v_0=\ord_{E_0}$ and $v_1=\ord_{E_1}$ matches the obvious line in $\QM(Y,E)$.

We divide this into two parts. First, we shall prove that the filtration $\fa_{\bullet,t}$ comes from a valuation; in other words, the induced degeneration has irreducible central fiber (see Lemma \ref{lem:irred degen -> valuation}). For this part, we observe that the induced degeneration can be decomposed into a two step degeneration by test configurations
\[
X\rightsquigarrow X_0\rightsquigarrow X_1,
\]
where the first degeneration is induced by $E_0$, while the second degeneration is induced by the specialization of $E_1$ on the degeneration $Y_0\to X_0$ of the Koll\'ar model $Y\to X$. Note that a priori $E_1$ may break into several components on $Y_0$, in which case $X_1$ will no longer be irreducible. Thus the key to this part of the proof is the following specialization result for Koll\'ar models (\cite{XZ-SDC}*{Section 4.2}).

\begin{prop}
Let $(Y,E)\to X$ be a Koll\'ar model. For any component $E'$ of $E$, let $(Y_0,E_0)\to X_0$ be the central fiber of the induced test configuration. Then:
\begin{enumerate}
    \item $X_0$ is klt.
    \item Each irreducible component of $E$ specializes to an irreducible component of $E_0$.
    \item $(Y_0,E_0)\to X_0$ is also a Koll\'ar model.
\end{enumerate}
\end{prop}

In particular, using this proposition we may conclude that $X_1$ is irreducible and even klt. The proof of the proposition itself is a delicate application of the tie-breaking method in birational geometry. In general, if $E$ has $r$ components, then the degeneration induced by $\fa_{\bullet,t}$ would be decomposed into $r$-steps and we need to apply the above proposition inductively, but overall the main idea stays the same.

To this end, we conclude that the geodesic joining Koll\'ar valuations (with respect to the same Koll\'ar model) lies in the valuations space. The next step is to show that the geodesic is the obvious line. This actually holds in a more general setting (see \cite{XZ-SDC}*{Lemma 4.8}):


\begin{prop}
Let $v_0,v_1\in \QM(Y,E)$ be quasi-monomial valuations in the same simplex, and let $(\fa_{\bullet,t})_{t\in [0.1]}$ be the geodesic connecting $v_0$ and $v_1$. Suppose that for some $t\in[0,1]$ the filtration $\fa_{\bullet,t}$ is induced by a valuation $w$. Then $($under some mild assumptions$)$ we have $w=(1-t)\cdot v_0+t\cdot v_1\in \QM(Y,E)$.
\end{prop}

In other words, whenever the geodesic intersects the valuations space, the intersection point is the obvious one in the corresponding simplex $\QM(Y,E)$. We remark that a priori it is not even clear why the intersection point lies in $\QM(Y,E)$. 


To apply the finite generation criterion from Theorem \ref{thm:fg criterion via km} to the $\hvol$-minimizer, we still need the next result.

\begin{thm} \label{thm:minimizer is from km}
Let $x\in X$ be a klt singularity and let $v_0\in \Val_{X,x}^*$ be the minimizer of the normalized volume function $\hvol$. Then $v_0$ is a Koll\'ar valuation. 
\end{thm}

If we assume that $\gr_v R$ is finitely generated, Theorem \ref{thm:minimizer is from km} can be deduced from the fact that the minimizer $v$ is the unique valuation that computes $\lct(\fab(v))$. In the divisorial case, this is exactly the argument we use in Section \ref{ss:div minimizer}. Since we don't know finite generation yet, we need to find a different argument, even in the divisorial case.

The idea is that many properties of Koll\'ar components are stable under perturbation. Conversely, we can detect whether a given divisor is a Koll\'ar component ``by perturbation''. As a typical example, we have the following characterization of Koll\'ar component.

\begin{lem} \label{lem:kc perturb boundary}
A prime divisor $E$ over a klt singularity $x\in X$ is a Koll\'ar component if and only if for any effective Cartier divisor $D$ on $X$, there exists some $\varepsilon\in \bQ_{>0}$ such that $E$ is an lc place of some complement of $(X,\varepsilon D)$.
\end{lem}

Essentially, we try to perturb the property that every Koll\'ar component is an lc place of some complement. The global version of this statement is \cite{Xu-K-stability-survey}*{Theorem 4.12}.

\begin{proof}
If $E$ is a Koll\'ar component and $Y\to X$ is the plt blowup, then $(Y,E)$ has a complement since the pair is plt and $-(K_Y+E)$ is ample. Taking pushforward we get a complement on $X$ with $E$ as an lc place. Both the plt and the ampleness condition are preserved if we add a small multiple of the strict transform of $D$ to the pair $(Y,E)$, thus the same statement holds for $(X,\varepsilon D)$. Conversely, if $E$ is an lc place of complement, then by \cite{BCHM} one can find a birational model $Y\to X$ with exceptional divisor $E$ such that $(Y,E)$ is lc and $-(K_Y+E)$ is ample over $X$. If these conditions are preserved after adding a small boundary, then $(Y,E)$ is in fact plt and thus $E$ becomes a Koll\'ar component.
\end{proof}

Similarly, we can formulate a higher rank analog that characterizes Koll\'ar valuations ``by perturbation''. 

\begin{prop} \label{prop:QM(km) perturb boundary}
Let $x\in X$ be a klt singularity and let $v\in \Val_{X,x}^*$ be a quasi-monomial valuation. Then the following are equivalent.
\begin{enumerate}
    \item The valuation $v$ is a Koll\'ar valuation. 
    \item For any effective Cartier divisor $D$ on $X$, there exists some $\varepsilon\in\bQ_{>0}$ such that $v$ is an lc place of some complement of $(X,\varepsilon D)$.
\end{enumerate}
\end{prop}

Unfortunately, the proof of the higher rank version is much harder. One reason is that we no longer have a canonical blowup that ``extract'' the higher rank valuation. The Koll\'ar models only serve as approximations of this extraction, and they are no longer Koll\'ar models if we add a small boundary divisor (e.g. the dlt condition may fail). This should not be surprising. After all, if we think of $\QM(Y,E)$ (where $(Y,E)$ is a Koll\'ar model) as a ``neighbourhood'' of the valuation $v$, then after perturbation, we should only expect properties of $v$ to hold after ``shrinking the neighbourhood'', which in practice means we should switch to a different Koll\'ar model.

To overcome these difficulties, we rely on the idea of \emph{special complements}, which was originally introduced in \cite{LXZ-HRFG} to attack the Fano version of the Higher Rank Finite Generation Conjecture.

\begin{defn}[special complements]
Let $x\in X$ be a klt singularity and let $\pi\colon (Y,E)\to X$ be a log smooth model. A \emph{special complement} of $x\in X$ (with respect to $(Y,E)$) is a complement $\Gamma$ such that $\pi_*^{-1}\Gamma\ge G$ for some effective ample $\bQ$-divisor $G$ on $Y$ whose support does not contain any stratum of $(Y,E)$.

Any valuation $v\in \QM(Y,E)\cap \LC(X,\Gamma)$ is called a \emph{monomial lc place} of the special complement $\Gamma$ (with respect to $(Y,E)$).
\end{defn}

Intuitively, the log smooth model is a log resolution of the Koll\'ar model, while conversely, the Koll\'ar model is the ample model of the log smooth model. The special complement condition can be regarded as a birational version of the log Fano condition in the definition of Koll\'ar models, while monomial lc places of special complements are the birational analog of monomial valuation on Koll\'ar models. The conditions are specifically designed so that the definition is not sensitive to the particular choice of the log smooth model nor the special complements. This offers room for perturbation. The proof of Proposition \ref{prop:QM(km) perturb boundary} now proceeds by showing that the two conditions in the statement are both equivalent to a third one:
\begin{enumerate}
    \setcounter{enumi}{2}
    \item The valuation $v$ is a monomial lc place of some special complement $\Gamma$ (with respect to some log smooth model $(Y,E)$).
\end{enumerate}
For details, see \cite{XZ-SDC}*{Section 3.3}.

Finally, we need to verify the equivalent conditions in Proposition \ref{prop:QM(km) perturb boundary} in order to finish the proof of Theorem \ref{thm:minimizer is from km}. This is accomplished by the following result.

\begin{prop}
Let $x\in X$ be a klt singularity and let $v_0\in \Val_{X,x}^*$ be the minimizer of the normalized volume function. Then for any effective Cartier divisor $D$ on $X$, there exists some $\varepsilon\in\bQ_{>0}$ such that $v_0$ is an lc place of some complement of $(X,\varepsilon D)$.
\end{prop}

This is proved in  \cite{XZ-SDC}*{Lemma 3.2}, and the argument in \emph{loc. cit.} also naturally gives an explicit value of $\varepsilon$. The proof essentially exploits the K-semistability property of the minimizer (Theorem \ref{thm:minimizer = K-ss val}). The idea is that asymptotically a K-semistable valuation is an lc place of their basis type divisors (which are asymptotically log canonical). If we choose a basis type divisor that maximizes the coefficient of $D$, and write the basis type divisor as $\varepsilon D+\Gamma$, then at least in the limit we get the desired coefficient $\varepsilon$ (some calculations are needed to show that it is positive) and the complement $\Gamma$. It remains to replace the limit by a finite level approximation, and this can be done using the ACC of log canonical threshold \cite{HMX-ACC}.

\subsection{K-polystable degeneration}

At this point, we have finished the proof of the Stable Degeneration Conjecture (Conjecture \ref{conj:SDC}). In particular, any klt singularity $x\in X=\Spec(R)$ has a degeneration to a K-semistable Fano cone singularity $x_0\in (X_0;\xi_0)$ induced by the minimizer $v_0\in \Val_{X,x}^*$ of the normalized volume function. To complete the two-step degeneration, it remains to construct the K-polystable degeneration of $x_0\in (X_0;\xi_0)$.


Before we sketch the ideas, it would be helpful to first review some argument in the case of vector bundles, where the analog of a polystable degeneration is the Jordan-H\"older filtration of a slope semistable vector bundle. The key to the existence of this filtration is the Schreier refinement theorem, stating that any two filtrations with semistable graded pieces of the same slopes have a common refinement. The Jordan-H\"older filtration is then obtained as the finest filtration of this kind. 

The construction of the K-polystable degeneration, see Theorem \ref{thm:K-ps degeneration}, is a generalization of this basic strategy. Its proof heavily relies on the following analog of the Schreier refinement theorem.

\begin{thm}[\cite{LWX-metric-tangent-cone}*{Theorem 4.1}] \label{thm:common K-ss degen}
Suppose that $x_i\in (X_i;\xi_i)\,(i=1,2)$ are two K-semistable degenerations of the Fano cone singularity $x_0\in (X_0;\xi_0)$. Then they have a common K-semistable degeneration $y\in (Y;\xi_Y)$.
\end{thm}

Assuming this result, let us explain how to construct the K-polystable degeneration. First note that Theorem \ref{thm:common K-ss degen} immediately implies the uniqueness of the K-polystable degeneration. Indeed, if $x_i\in (X_i;\xi_i)\,(i=1,2)$ are two K-polystable degenerations of the Fano cone singularity $x_0\in (X_0;\xi_0)$, then they have a common K-semistable degeneration $y\in (Y;\xi_Y)$. However, because both $x_i\in (X_i;\xi_i)$ are K-polystable, their K-semistable degenerations are isomorphic to themselves. Thus we get 
\[
\big(x_1\in (X_1;\xi_1)\big) \cong \big(y\in (Y;\xi_Y)) \cong \big(x_2\in (X_2;\xi_2)\big),
\]
which gives the uniqueness.

Next we prove the existence of the K-polystable degeneration. Suppose that $x_0\in (X_0;\xi_0)$ is not K-polystable. By definition, this means that there exists a K-semistable degeneration $x_1\in (X_1;\xi_1)$ that is not isomorphic to $x_0\in (X_0;\xi_0)$. If $x_1\in (X_1;\xi_1)$ is still not K-polystable, we can find a further degeneration $x_2\in (X_2;\xi_2)$ and continue. The key is to make this process stop after finitely many steps. 

A discrete invariant that grows under this procedure is the dimension of the maximal torus. Note that the automorphism group $\Aut(x\in (X;\xi))$ of a Fano cone singularity (i.e. the group of $\langle \xi\rangle$-equivariant automorphisms of the singularity $x\in X$) is an algebraic group. We denote by $\bT_i$ the maximal torus of $\Aut(x_i\in (X_i;\xi_i))$, which is well-defined up to conjugation.

\begin{claim*}
If $x_i\in (X_i;\xi_i)$ is not K-polystable, then it has a $\bT_i$-equivariant K-semistable degeneration $x_{i+1}\in (X_{i+1};\xi_{i+1})$ that is not isomorphic to $x_i\in (X_i;\xi_i)$. Moreover, $\dim \bT_{i+1} > \dim \bT_i$.
\end{claim*}

The second part of the claim actually follows from the first, since we clearly have $\dim \bT_{i+1} \ge \dim \bT_i$, and through the graded algebra description of $X_{i+1}$ we get an additional $\bG_m$-action on $x_{i+1}\in (X_{i+1};\xi_{i+1})$ from the grading. Since the dimension of the maximal torus is at most the dimension of the singularity (otherwise the torus action is not effective), the claim implies that the K-semistable degeneration process necessarily stops after finitely many steps.

It remains to construct the equivariant K-semistable degeneration in the above claim. We start with any  test configuration that degenerates $x_i\in (X_i;\xi_i)$ to some non-isomorphic K-semistable Fano cone singularity $x_{i+1}\in (X_{i+1};\xi_{i+1})$. The idea is to use Theorem \ref{thm:common K-ss degen} to find a ``toric degeneration'' of this test configuration. Note that for any one parameter subgroup $\rho\colon \bG_m\to \bT_i$, we also have a product test configuration of $x_i\in (X_i;\xi_i)$ induced by the weight filtration of the $\bG_m$-action. The central fiber of the product test configuration is just $x_i\in (X_i;\xi_i)$ itself. By Theorem \ref{thm:common K-ss degen}, we have a common degeneration as illustrated by the following diagram
\[
\xymatrix{
x_i\in (X_i;\xi_i) \ar@{~>}^{\bG_m}[r] \ar@{~>}[d] & x_i\in (X_i;\xi_i) \ar@{~>}[d] \\
x_{i+1}\in (X_{i+1};\xi_{i+1}) \ar@{~>}[r] & y\in (Y,\xi_Y).
}
\]
In some sense we view the right column as the ``toric degeneration'' of the left column. In fact, in the proof of Theorem \ref{thm:common K-ss degen}, this is what happens at the level of filtrations. By construction, the degeneration
\[
x_i\in (X_i;\xi_i) \rightsquigarrow y\in (Y,\xi_Y)
\]
is equivariant with respect to the chosen one parameter subgroup $\rho(\bG_m)$. It also inherits the torus action on the original degeneration 
\[
x_i\in (X_i;\xi_i) \rightsquigarrow x_{i+1}\in (X_{i+1};\xi_{i+1}).
\]
If we replace $x_{i+1}\in (X_{i+1};\xi_{i+1})$ by $y\in (Y,\xi_Y)$ and repeat this construction for a finite collection of one parameter subgroups that generate $\bT_i$, we will eventually get the desired $\bT_i$-equivariant K-semistable degeneration. This proves the claim and we have finished the construction of the K-polystable degeneration assuming Theorem \ref{thm:common K-ss degen}.

We now return to sketch a proof of Theorem \ref{thm:common K-ss degen}. Recall that test configurations of $x_0\in X_0=\Spec(R_0)$ are given by filtrations of $R_0$. In particular, there are filtrations $\fa_{\bullet,i}$ ($i=1,2$) whose associated graded algebra gives $X_i$, i.e., 
\[
X_i=\Spec(\gr_{\fa_{\bullet,i}} R_0).
\]
These filtrations have equivalent refinements, namely, each filtration $\fa_{\bullet,i}$ induces a filtration on the associated graded algebra $\gr_{\fa_{\bullet,j}} R_0$ of the other, and the induced filtrations satisfy
\[
\gr_{\fa_{\bullet,2}} \gr_{\fa_{\bullet,1}} R_0 \cong \gr_{\fa_{\bullet,1}} \gr_{\fa_{\bullet,2}} R_0.
\]
Denote this (doubly graded) algebra by $R'$. Then $\Spec(R')$ is the obvious candidate of the common degeneration. 

To make this strategy work, we need to show that $R'$ is finitely generated. Note that both filtrations $\fa_{\bullet,i}$ are induced by some \emph{divisorial} valuations $v_i=\ord_{E_i}$, and we may realize $R'$ as a quotient of the Cox ring
\[
\bigoplus_{m_1,m_2\in\bN} \pi_*\cO_Y(-m_1 E_1-m_2 E_2),
\]
where $\pi\colon Y\to X$ is a birational model that extracts both divisors $E_i$. The general results from \cite{BCHM} tell us that this Cox ring is finitely generated if we can find an effective $\bQ$-divisor $D$ on $X$ such that $(X,D)$ is lc and $A_{X,D}(E_i)<1$ for both $i$, because these conditions imply that the two divisors $E_i$ can be simultaneously extracted on a model $Y$ that is of Fano type over $X$, and Cox rings on Fano type varieties are finitely generated by \cite{BCHM}.

We haven't used the assumption that both valuations $v_i$ induce K-semistable degenerations of $x_0\in (X_0;\xi_0)$. It turns out that this condition is equivalent to the vanishing of the generalized Futaki invariant, or
\begin{equation} \label{e:A=S for K-ss degen}
    A_{X_0}(v_i)=S(\wt_{\xi_0};v_i),
\end{equation}
see \cite{LWX-metric-tangent-cone}*{Lemma 3.1 and proof of Theorem 4.1}. In other words, if we consider $\bT=\langle\xi_0 \rangle$-invariant basis type divisors that are compatible with $v_i$\footnote{Namely, the corresponding basis is $\bT=\langle\xi_0 \rangle$-invariant and is compatible with the filtration induced by $v_i$. Such basis type divisors maximize the vanishing order along $v_i$ and therefore asymptotically compute $S(\wt_{\xi_0};v_i)$.}, then asymptotically they are log canonical (because the Fano cone singularity $x_0\in (X_0;\xi_0)$ is K-semistable) and have the valuation $v_i$ as an lc place (because of the identity \eqref{e:A=S for K-ss degen}). By choosing basis type divisors that are simultaneously compatible with both $v_i$\footnote{Given two filtrations on a vector space, there is always a simultaneously compatible basis. See e.g. \cite{AZ-K-adjunction}*{Lemma 3.1} or \cite{BE-compatible-basis}*{Proposition 1.14}).}, in the limit we would get the desired auxiliary divisor $D$.

It then follows from the previous discussion that $R'$ is finitely generated, and we get a common degeneration to $y\in (Y:=\Spec(R');\xi_Y)$, where $\xi_Y$ is the induced Reeb vector. It remains to check that the Fano cone singularity $y\in (Y,\xi_Y)$ is K-semistable. Roughly speaking, this is because the degenerations are induced by lc places of basis type divisors, hence the degenerations of basis type divisors remain log canonical. Alternatively, it follows from the vanishing of the generalized Futaki invariants, a property that passes on to the induced degenerations $x_i\in (X_i;\xi_i) \rightsquigarrow y\in (Y,\xi_Y)$. 

\section{Boundedness of singularities} \label{s:bdd}

One of the recent achievements in K-stability of Fano varieties is the construction of the K-moduli space, a proper moduli space that parametrizes K-polystable Fano varieties. A detailed account on this topic is \cite{Xu-K-book}. Among other things, the content of the K-moduli theorem can be summarized as follows.

\begin{thm} \label{thm:K-moduli}
For any positive integer $n$ and any positive real number $\varepsilon$, there exists a projective moduli space parametrizing K-polystable Fano varieties of dimension $n$ and anti-canonical volume at least $\varepsilon$.
\end{thm}

There should be a local analog of the K-moduli for klt singularities. In general, klt singularities may have infinite dimensional deformation spaces, so we certainly need to restrict the class of singularities we consider in order to have a reasonably behaved moduli space. The Stable Degeneration Conjecture and the surrounding stability theory of klt singularities suggest the following refinement of the local-to-global correspondence.

\medskip

\begin{tabularx}{0.9\textwidth} { 
  >{\centering\arraybackslash}X 
  | >{\centering\arraybackslash}X 
  }
    global & local \\
    \hline
    K-semi/polystable Fano varieties $V$ & K-semi/polystable Fano cone singularities $x\in (X;\xi)$ \\
    \hline
    anti-canonical volume $(-K_V)^{\dim V}$ & local volume $\hvol(x,X)$
\end{tabularx}

\medskip

In particular, it seems reasonable to expect that for any positive integer $n$ and any real number $\varepsilon>0$, there exists a projective moduli space parametrizing K-polystable Fano cone singularities of dimension $n$ and anti-canonical volume at least $\varepsilon$. While many parts of K-moduli theory should carry over to the local setting, the boundedness part remains quite mysterious. In this section, we discuss what is known so far about boundedness of klt singularities and what are the challenges.

We say that a given set $\cS$ of klt singularities is \emph{bounded} if there exists a $\bQ$-Gorenstein family $B\subseteq \cX\to B$ of klt singularities such that every singularity $x\in X$ in the set $\cS$ is isomorphic to some fiber of $B\subseteq \cX\to B$\footnote{We do not require that all the fibers belong to the set $\cS$.}. For boundedness of Fano cone singularities, we will also require that there is a fiberwise torus action on the family $B\subseteq \cX\to B$ such that the Reeb vectors lie in the Lie algebra of the acting torus\footnote{It is quite likely that this additional condition is automatic once the underlying singularities are bounded. This is related to the following question: suppose $\cX\to B$ is a flat (but not necessarily projective) family of algebraic varieties and assume that there is a Zariski dense subset $B_0$ of $B$ such that the fibers over $B_0$ admit a $\bT$-action for some fixed torus $\bT$; then does the family $\cX\to B$ have a fiberwise torus action, possibly after replacing $B$ by a dense open subset?}. The following is a more precise formulation of the Boundedness Conjecture for Fano cone singularities.

\begin{conj}[\cite{XZ-SDC}*{Conjecture 1.7}] \label{conj:bdd}
Fix a positive integer $n$ and some real number $\varepsilon>0$. Then the set of $n$-dimensional K-semistable Fano cone singularities $x\in (X;\xi)$ with local volume $\hvol(x,X)\ge \varepsilon$ is bounded.
\end{conj}

A variant of this conjecture is the special boundedness conjecture \cite{HLQ-vol-ACC}, whose weaker version predicts that $n$-dimensional klt singularities $x\in X$ with $\hvol(x,X)\ge \varepsilon$ are bounded up to special degenerations. There is also a related ACC conjecture for local volumes, stating that in any fixed dimension, the set of possible local volumes are discrete away from zero. They both follow from the Stable Degeneration Conjecture and the boundedness conjecture above, as the stable degeneration of a klt singularity preserves the local volume (Theorem \ref{thm:SDC K-ss part}). For the ACC conjecture, we also need the constructibility of local volumes in $\bQ$-Gorenstein families (Theorem \ref{thm:constructible}).

It is not hard to verify the Boundedness Conjecture in dimension two. In fact, klt singularities in dimension two are the same as quotient singularities, and K-semistable Fano cone singularities are the linear quotients, i.e. they are isomorphic to $\bC^2/G$ for some finite group $G\subseteq GL(2,\bC)$ that does not contain any pseudoreflections. By the finite degree formula (Theorem \ref{thm:finite deg formula}), we see that their local volume is $\frac{4}{|G|}$, hence there are only finitely many isomorphism classes if the local volume is bounded away from zero. When the dimension is at least three, a full classification of klt singularities is no longer available, and the Boundedness Conjecture becomes much harder.

Let us also draw some comparison with the corresponding boundedness result for Fano varieties, which is also part of the K-moduli theorem (Theorem \ref{thm:K-moduli}).

\begin{thm}[\cites{Jia-boundedness,LLX-nv-survey,XZ-minimizer-unique}] \label{thm:Fano bdd}
Fix a positive integer $n$ and some real number $\varepsilon>0$. Then the set of $n$-dimensional K-semistable Fano variety $V$ with volume $(-K_V)^n\ge \varepsilon$ form a bounded family.
\end{thm}

Consider the special case of the Boundedness Conjecture concerning orbifold cones $o\in C_a(V,L)$. From Example \ref{exp:cone minimizer}, we know that the orbifold cone singularity is K-semistable if and only if the base $V$ is a K-semistable Fano variety. While it is very tempting to relate the Boundedness Conjecture in this case to the boundedness of Fano varieties, a direct computation shows that the local volume $\hvol(o,C_a(V,L))$ is only a multiple of the anti-canonical volume of $V$. Namely, we have
\[
\hvol(o,C_a(V,L)) = r\cdot \vol(-K_V)
\]
where $r>0$ is the rational number satisfying $-K_V\sim_\bQ rL$. The largest possible value of $r$ we can get as we vary the Weil divisor $L$ is called the Weil index of the Fano variety $V$. Note that the Weil index of a Fano variety can be arbitrarily big in a fixed dimension; for example, the Weil index of the weighted projective space $\bP(a_0,a_1,\dots,a_n)$ is $a_0+a_1+\dots+a_n$. Even if we assume that the Fano variety is K-semistable, there does not seem to be any particular reason for the Weil index to be bounded. Thus already in this special case, it is not clear how to deduce Conjecture \ref{conj:bdd} from Theorem \ref{thm:Fano bdd}. In some sense, the presence of the Weil index is one of the major difficulties in the study of the Boundedness Conjecture for klt singularities.

Some partial progress on the Boundedness Conjecture have been made in \cites{HLQ-vol-ACC,MS-bdd-toric,LMS-bdd-dim-3,Z-mld^K-1,Z-mld^K-2}. In particular, the conjecture is known for hypersurface singularities, for threefold singularities, and for singularities of complexity at most one. The works \cites{Z-mld^K-1,Z-mld^K-2} also introduce an approach to the Boundedness Conjecture through the minimal log discrepancies of Koll\'ar components.

\begin{defn}[\cite{Z-mld^K-1}]
Let $x\in X$ be a klt singularity. The minimal log discrepancy of Koll\'ar components, denoted $\mldk (x,X)$, is the smallest log discrepancy $A_X(E)$ as $E$ varies among all Koll\'ar components over $x\in X$.
\end{defn}

One of the main results of \cites{Z-mld^K-1,Z-mld^K-2} is the following boundedness criterion.

\begin{thm} \label{thm:K-ss cone bdd criterion}
Fix a positive integer $n$ and consider a set $\cS$ of $n$-dimensional K-semistable Fano cone singularities. Then $\cS$ is bounded if and only if there exist some $\varepsilon,A>0$ such that
\[
\hvol(x,X)\ge \varepsilon \quad \mathit{and} \quad \mldk(x,X)\le A
\]
for all $x\in (X;\xi)$ in $\cS$.
\end{thm}

The idea of the proof comes from the following observation. Given a K-semistable Fano cone singularity $x\in (X;\xi)$, each rational Reeb vector on $X$ induces a projective orbifold cone compactification $\oX$ of $X$. As the Reeb vector approximates the K-semistable polarization $\xi$, the volumes $\vol(-(K_{\oX}+D))$ of the log Fano pair $(\oX,D)$ (where $D=\oX\setminus X$ is the divisor at infinity) approximates the local volume $\hvol(x,X)$ of the singularity. In particular, the anti-canonical volume of $\oX$ is bounded. One should note that the compactification $\oX$ is not unique, and in general not bounded, as illustrated by the following example.

\begin{expl}
Let $a_1,\dots,a_n\in \bN^*$ be pairwise coprime integers. Then $\xi=(a_1,\dots,a_n)\in \bN^n$ gives a polarization of the Fano cone singularity $0\in\bA^n$; it generates the $\bG_m$-action with weights $a_1,\dots,a_n$ on the coordinates. This endows $\bA^n$ with an affine orbifold cone structure $C_a(V,L)$ where $V=\bP(a_1,\dots,a_n)$ and $L=\cO_V(1)$. The associated projective orbifold cone is $\oX=\bP(1,a_1,\dots,a_n)$, which do not form a bounded family as the weights $a_i$'s vary. 
\end{expl}

Nonetheless, one can still extract some weaker boundedness in the above example: at least the linear system $|-K_{\oX}|$ always defines a birational map that is an embedding at the vertex $[1:0:\dots:0]$. In fact, if $[s:x_1:\dots:x_n]$ are the weighted homogeneous coordinates of $\oX$, then for every $i\in \{1,\dots,n\}$ there exists some $k_i\in \bN$ such that $s^{k_i} x_i\in H^0(-K_{\oX})$ (this is possible because $s$ has weight $1$); it is not hard to see that the sub linear system spanned by $s^{k_i} x_i$ ($i=0,\dots,n$) is base point free and restricts to an embedding on the affine chart $\bA^n=\oX\setminus (s=0)$.

In general, we have an effective birationality result (\cite{Z-mld^K-2}*{Proposition 3.8}): there exists a positive integer $m$ depending only on $\hvol(x,X)$ and $\mldk(x,X)$ such that $|-mK_{\oX}|$ induces a birational map that restricts to an embedding on $X$. This implies the boundedness of the Fano cone singularity $x\in (X;\xi)$.

In some situations, one can use classification results to verify the boundedness of $\mldk$ and hence prove the Boundedness Conjecture \ref{conj:bdd} using Theorem \ref{thm:K-ss cone bdd criterion}. This is the case for singularities of complexity at most one \cite{LMS-bdd-dim-3}, and for threefold singularities whose local volumes are bounded away from zero \cite{Z-mld^K-1}. However, in general it is not yet clear what to expect about the behaviour of $\mldk$. 

\section{Questions and future directions} \label{s:question}

In this last section, we collect some conjectures and open questions about the stability and boundedness of klt singularities, hoping that the readers will become motivated to work on
some of them.

\subsection{Boundedness}

One of the major challenges in this topic is the Boundedness Conjecture (Conjecture \ref{conj:bdd}). We restate it here for the readers' convenience.

\begin{conj}[Boundedness]
Fix a positive integer $n$ and some real number $\varepsilon>0$. Then the set of $n$-dimensional K-semistable Fano cone singularities $x\in (X;\xi)$ with local volume $\hvol(x,X)\ge \varepsilon$ is bounded.
\end{conj}

As discussed in Section \ref{s:bdd}, the Boundedness Conjecture has several interesting consequences. Some of these might be easier to study. The first one is the Special Boundedness Conjecture.

\begin{conj}[Special Boundedness]
Fix a positive integer $n$ and some real number $\varepsilon>0$. Then the set of $n$-dimensional klt singularities $x\in X$ with local volume $\hvol(x,X)\ge \varepsilon$ is bounded up to special degenerations.
\end{conj}

Since the multiplicity and embedded dimension of a singularity are bounded in a given family and are non-decreasing under specialization, another consequence of the Boundedness Conjecture is the boundedness of these invariants.

\begin{conj}
Fix a positive integer $n$ and some real number $\varepsilon>0$. Then there exists some constant $M$ depending only on $n,\varepsilon$ such that for all $n$-dimensional klt singularities $x\in X$ with local volume $\hvol(x,X)\ge \varepsilon$, we have $\mult_x X\le M$ and the embedded dimension of $x\in X$ is also at most $M$.
\end{conj}

Apart from the stable degenerations, there are other ways to produce special degenerations of klt singularities. Essentially, special test configurations are in one-to-one correspondence with Koll\'ar components. By the Borisov-Alexeev-Borisov Conjecture (now a theorem of Birkar \cites{Birkar-bab-1,Birkar-bab-2}), the Koll\'ar components (viewed as log Fano varieties) belong to a bounded family if and only if their minimal log discrepancies (mld) are bounded away from zero. This motivates a stronger version of the Special Boundedness Conjecture (see \cite{HLQ-vol-ACC}*{Conjecture 1.6}).

\begin{conj} \label{conj:delta-plt blowup}
Fix a positive integer $n$ and some real number $\varepsilon>0$. Then there exists some constant $\delta>0$ depending only on $n,\varepsilon$ such that any $n$-dimensional klt singularities $x\in X$ with local volume $\hvol(x,X)\ge \varepsilon$ admits a $\delta$-plt blowup.
\end{conj}

Here we say the klt singularity $x\in X$ admits a $\delta$-plt blowup if there exists a plt blowup $Y\to X$ of a Koll\'ar component $E$ such that the pair $(Y,E)$ is $\delta$-plt, i.e., $A_{Y,E}(F)>\delta$ for all prime divisors $F$ that are exceptional over $Y$. Conjecture \ref{conj:delta-plt blowup} has been verified up to dimension three, see \cites{HLQ-vol-ACC,Z-mld^K-1}. 

We remark that while Conjectures \ref{conj:bdd} and \ref{conj:delta-plt blowup} both imply the Special Boundedness Conjecture, it is not clear whether any of these two implies the other.

Another consequence of the Boundedness Conjecture is the discreteness of the local volume away from zero. Sometimes this is also referred to as the ACC Conjecture for local volumes\footnote{If we consider klt pairs $(X,D)$ with DCC coefficients, then their local volumes are expected to form an ACC set, hence the name of the conjecture.}. See \cite{LX-cubic-3fold}*{Question 4.3} and \cite{LLX-nv-survey}*{Question 6.12}.

\begin{conj}[ACC] \label{conj:ACC}
Fix a positive integer $n$. Then the set of all possible local volumes of $n$-dimensional klt singularities are discrete away from zero.
\end{conj}

By Theorem \ref{thm:largest vol}, the largest local volume is achieved by a smooth point. Assuming the above ACC conjecture, a natural question is what should be the second largest local volume. A natural prediction is given by the ODP volume gap conjecture, see \cite{SS-two-quadric}*{Conjecture 5.5} and \cite{LX-cubic-3fold}*{Conjecture 4.5}.

\begin{conj}[ODP volume gap]
The second largest volume of an $n$-dimensional klt singularity is $2(n-1)^n$, and it is achieved only by the ordinary double point.
\end{conj}

By \emph{loc. cit.}, this conjecture implies that the K-moduli space of cubic hypersurfaces coincides with their GIT moduli space. On the other hand, the existence of a volume gap already seems nontrivial.

\begin{conj}
There exists some constant $\varepsilon>0$ such that the only $n$-dimensional klt singularity with local volume at least $(1-\varepsilon)n^n$ is the smooth point.
\end{conj}

The ODP volume gap conjecture also has a global analog. Recall that by Theorem \ref{thm:Fano bdd}, the volumes of K-semistable Fano varieties are known to be discrete away from zero. A theorem of Fujita \cite{Fuj-largest-vol-Pn} says that the projective space has the largest volume among them in any fixed dimension.

\begin{conj}[Second largest volume] \label{conj:K-ss 2nd vol}
The second largest anti-canonical volume of an $n$-dimensional K-semistable Fano variety is $2n^n$, and it is achieved only by $\bP^1\times \bP^{n-1}$ and the smooth quadric hypersurface in $\bP^{n+1}$.
\end{conj}

An interesting (but also mysterious) feature of the global version is that there are two Fano varieties with second largest volume. On the other hand, because one of them, $\bP^1\times \bP^{n-1}$, is toric, the toric case of Conjecture \ref{conj:K-ss 2nd vol} is also interesting by its own. It might be approachable using combinatorial argument and will provide further evidence for Conjecture \ref{conj:K-ss 2nd vol}.

\begin{conj}
Among $n$-dimensional K-semistable toric Fano variety, $\bP^1\times \bP^{n-1}$ has the second largest anti-canonical volume.
\end{conj}

Going back to the Boundedness Conjecture, if we compare it with Theorem \ref{thm:K-ss cone bdd criterion}, we are naturally led to the following speculation, see \cite{Z-mld^K-1}*{Conjecture 1.7}.

\begin{conj} \label{conj:mld^K bdd if vol>=epsilon}
Fix a positive integer $n$ and some real number $\varepsilon>0$. Then there exists some constant $A$ depending only on $n,\varepsilon$ such that 
\[
\mldk(x,X)\le A
\]
for any $n$-dimensional klt singularity $x\in X$ with $\hvol(x,X,\Delta)\ge \varepsilon$.
\end{conj}

Shokurov has conjectured that the set of minimal log discrepancies (mld) satisfies the ACC \cite{Sho-mld-conj}. In particular, there should be an upper bound on the mlds that only depends on the dimension. This is known as the boundedness (BDD) conjecture for mld. By analogy, we are tempted to ask whether the same holds for $\mld^K$, and in particular, whether the lower bound on the local volume is really necessary in Conjecture \ref{conj:mld^K bdd if vol>=epsilon}.

\begin{que}[ACC and BDD for $\mldk$]
Fix a dimension $n$. Does the set of $\mldk$ of $n$-dimensional klt singularities satisfy the ACC? Is there a constant $A$ depending only on $n$ such that 
\[
\mldk(x,X)\le A
\]
for any $n$-dimensional klt singularity $x\in X$?
\end{que}

Perhaps what makes this question hard to study is the lack of understanding for Koll\'ar components that minimizes the log discrepancy function.

\begin{que}
Is there an intrinsic way to tell whether a given Koll\'ar component computes $\mldk$?
\end{que}

There are some klt singularities with unique Koll\'ar component. They are characterized by the property that the induced log Fano pair $(E,\Delta_E)$ on the Koll\'ar component $E$ (see Section \ref{s:kc}) is weakly special, i.e. $(E,\Delta_E+D)$ is log canonical for any effective $\bQ$-divisor $D\sim_\bQ -(K_E+\Delta_E)$. Consider orbifold cones over weakly special Fano varieties as a special case. Their $\mldk$ are closely related to the Weil index\footnote{Recall that the Weil index of a Fano variety $V$ is the largest integer $q$ such that $-K_V\sim_\bQ qL$ for some Weil divisor $L$ on $V$.} of the Fano varieties. We may ask:

\begin{que}
Fix a dimension $n$. Can the Weil index of $n$-dimensional weakly special Fano varieties be arbitrarily big?
\end{que}

\subsection{Local volumes}

The local volume is a delicate invariant of a klt singularity, and it is still quite mysterious how it behaves under the steps of the minimal model program, especially flips.

\begin{que}
Does the local volume satisfy some type of monotonicity under flips?
\end{que}

It is not clear what kind of monotonicity should be there. On one hand, since the flip improves the singularity in general, we may hope that the local volume increases under flips. On the other hand, one can also find toric flips $X_1\dashrightarrow X_2$ such that
\[
\min_{x_1\in X_1}\hvol(x_1,X_1)>\min_{x_2\in X_2}\hvol(x_2,X_2).
\]
It is possible that the correct formulation of the monotonicity should involve some motivic version of the local volumes.

The local volumes are also expected to relate to singularity invariants in positive characteristics. Given a klt singularity $x\in X$ in characteristic $0$, we may consider its reduction $x_p\in X_p$ modulo a prime $p$. From \cites{Har-klt=F-reg,HW-klt-pair-F-reg}, we know that the mod $p$ reduction $x_p\in X_p$ is strongly $F$-regular when $p\gg 0$. An interesting invariant of a strongly $F$-regular singularity $x\in X$ in positive characteristic is its $F$-signature $s(x,X)$ (see \cites{HL-F-signature-def,Tuc-F-signature-exist}), and a folklore question in commutative algebra is to find geometric interpretations of $\lim_{p\to \infty} s(x_p,X_p)$. Partly motivated by this question, a comparison result between the local volume and the $F$-signature is conjectured in \cite{LLX-nv-survey}*{Section 6.3.1}. Here we state a modified version.

\begin{conj} \label{conj:compare F-sign and vol}
For any $n$-dimensional klt singularity $x\in X$ in characteristic $0$, let $x_p\in X_p$ be its reduction mod $p\gg 0$. Then
\[
\liminf_{p\to\infty} s(x_p,X_p)\ge \frac{\hvol(x,X)}{n^n}.
\]
\end{conj}

The right hand side is also known as the volume density of the singularity. It is not hard to see that the inequality becomes an equality when $x\in X$ is smooth. A weaker conjecture would replace the constant $n^n$ by the existence of some positive dimensional constant. If the (weaker) conjecture is true, it will give a positive answer to \cite{CRST-F-reg-pi-1}*{Question 5.9}, which asks whether the $F$-signatures $s(x_p,X_p)$ have uniform lower bounds as $p\to\infty$.

One motivation for Conjecture \ref{conj:compare F-sign and vol} is the finite degree formula for $F$-signature, which is reminiscent of the finite degree formula for local volumes (Theorem \ref{thm:finite deg formula}).

\begin{thm}[\cite{CRST-F-reg-pi-1}*{Theorem B}]
Let $f\colon (y\in Y)\to (x\in X)$ be a finite quasi-\'etale morphism between strongly $F$-regular singularities. Then 
\[
s(y,Y) = \deg(f) \cdot s(x,X). 
\]
\end{thm}

Note that \cite{CRST-F-reg-pi-1}*{Theorem 4.4} proves a much more general finite degree formula for crepant morphisms between strongly $F$-regular pairs. In contrast, the finite degree formula for local volumes is currently restricted to \emph{Galois} morphisms. It would be necessary to resolve this discrepancy.

\begin{conj}
Let $f\colon \big(y\in (Y,\Delta_Y)\big)\to \big(x\in(X,\Delta)\big)$ be a finite surjective morphism between klt pairs such that $f^*(K_X+\Delta)=K_Y+\Delta_Y$. Then
\[
\hvol(y,Y,\Delta_Y) = \deg(f) \cdot \hvol(x,X,\Delta).
\]
\end{conj}

One obvious subtlety is that if we pass to the Galois closure, the boundary divisor $\Delta_Y$ may have negative coefficients. Perhaps there is some possibility of developing a stability theory for sub-pairs.

Guided by Conjecture \ref{conj:compare F-sign and vol}, it seems reasonable to believe that many nice properties of the local volume (and even the stability theory itself) carry over to positive  characteristics. For example, one can ask:

\begin{que}
Fix a dimension $n$ and some real number $\varepsilon>0$. Is the set of strongly $F$-regular singularities $x\in X$ (in characteristic $p$) with $F$-signature $s(x,X)\ge \varepsilon$ bounded up to special degenerations?
\end{que}

\begin{que}
In a fixed dimension $n$ and characteristic $p$, are the possible values of $F$-signatures discrete away from zero? What is the second largest $F$-signature?  
\end{que}
 


\subsection{Miscellaneous}

There are some basic properties of the normalized volume function that are still not fully understood. The following question is taken from \cite{Li-normalized-volume}.

\begin{que}
Is the normalized volume function lower semi-continuous on the valuation space?
\end{que}

A formal arc through a singularity $x\in X$ is a morphism $\phi\colon \Spec (\bk[\![t]\!])\to X$ such that $\phi(0)=x$. The arc space of the singularity, which parameterizes the formal arcs, is an essential tool in the theory of motivic integration and is also quite useful in the study of invariants in birational geometry, see e.g. \cites{Mus-jet-scheme-cpi,Mus-lct-jet-scheme,EMY-mld-arc-sp,ELM-arc-sp-contact-loci} for some applications of this kind. A natural question (communicated to us by Chenyang Xu) is whether the local volumes of singularities have interpretations through the arc space. Note that \cite{dFM-vol-arc} have defined volumes for subsets of the arc space.

\begin{que}
Can the local volume be defined using invariants of the arc space?
\end{que}

\begin{rem}
In a somewhat related direction, one can also ask whether the local volume only depends on the contact geometry of the link of the klt singularity. The answer is no in general. The reason is that there are smooth families of Fano manifolds whose general fibers are K-semistable while some special fibers are K-unstable (one explicit example is the family $2.26$ of Fano threefolds, see \cite{ACC+-Fano3}*{Section 5.10}). By Example \ref{exp:cone minimizer}, this implies that the local volume is not constant on the corresponding family of cones. On the other hand, since the original family is smooth, the fibers are symplectomorphic and therefore the links of the cone singularities have isomorphic contact structure.
\end{rem}

On Fano varieties, \cite{LXZ-HRFG}*{Theorem 4.5} relates the finite generation property of lc places of complements to the linearity of the $S$-invariants. In the applications to explicit examples, this is the easiest way to check finite generation. There might be a local analog.

\begin{que}
Find finite generation criterion (possibly in terms of $S$-invariant or other geometric conditions) for general lc places of complements of a klt singularity.
\end{que}

Since Koll\'ar valuations are the higher rank versions of Koll\'ar components, we may ask whether some of the known properties of Koll\'ar components have higher rank analog. For example:

\begin{que}
For any graded sequence $\fab$ of $\fm_x$-primary ideals on a klt singularity $x\in X$, is the log canonical threshold $\lct(\fab)$ always computed by some Koll\'ar valuation?
\end{que}







\bibliography{ref}

\end{document}